\documentclass{amsart}
\usepackage{amsmath, amsthm, amssymb, amscd, epsfig, eucal}

\begin{document}

\newtheorem{theorem}{Theorem}[section]
\newtheorem{lemma}[theorem]{Lemma}
\newtheorem{proposition}[theorem]{Proposition}
\newtheorem{corollary}[theorem]{Corollary}
\newtheorem{conjecture}[theorem]{Conjecture}
\newtheorem{question}[theorem]{Question}
\newtheorem{problem}[theorem]{Problem}
\newtheorem*{claim}{Claim}
\newtheorem*{criterion}{Criterion}
\newtheorem*{hyperbolic_thm}{Hyperbolic Geodesic Theorem~\ref{hyperbolic_geodesic_theorem}}
\newtheorem*{upper_bound_thm}{Universal Upper Bound Theorem~\ref{universal_upper_bound_random_walk}}
\newtheorem*{hyperbolic_lb_thm}{Hyperbolic Lower Bound Theorem~\ref{thm:hyperbolic_lower_bound}}
\newtheorem*{reducible_cor}{Reducible Bound Corollary~\ref{cor:reducible_lower_bound}}
\newtheorem*{growth_obstruction_cor}{Growth Obstruction Corollary~\ref{cor:growth_obstruction}}
\newtheorem*{random_walk_thm}{Random Walk Theorem~\ref{random_walk_theorem}}
\newtheorem*{linear_translation_thm}{Linear Translation Length Theorem~\ref{theorem:linear_translation}}
\newtheorem*{area_winding_thm}{Area and Winding Theorem~\ref{random_area_theorem}}

\newtheorem*{rotation_clt_thm}{Rotation Number Central Limit Theorem~\ref{rotation_clt}}

\theoremstyle{definition}
\newtheorem{definition}[theorem]{Definition}
\newtheorem{construction}[theorem]{Construction}
\newtheorem{notation}[theorem]{Notation}

\theoremstyle{remark}
\newtheorem{remark}[theorem]{Remark}
\newtheorem{example}[theorem]{Example}

\numberwithin{equation}{subsection}

\newcommand\id{\textnormal{id}}

\newcommand\Z{\mathbb Z}
\newcommand\R{\mathbb R}
\newcommand\Q{\mathbb Q}
\renewcommand\H{\mathbb H}
\renewcommand\O{\mathcal O}

\renewcommand\Pr{{\mathbf P}}
\newcommand\Ex{{\mathbf E}}

\newcommand\cl{\textnormal{cl}}
\newcommand\scl{\textnormal{scl}}
\newcommand{\con}[1]{C_{#1}}
\newcommand{\length}{\textnormal{length}}
\newcommand{\area}{\textnormal{area}}
\newcommand{\wind}{\textnormal{wind}}
\newcommand{\eval}{\textnormal{eval}}
\newcommand{\word}{\textnormal{word}}
\newcommand{\path}{\textnormal{path}}
\newcommand{\cone}{\textnormal{cone}}
\newcommand{\homeo}{\textnormal{Homeo}}
\newcommand{\rot}{\textnormal{rot}}
\newcommand{\PSL}{\textnormal{PSL}}
\newcommand{\MCG}{\textnormal{MCG}}
\newcommand{\Out}{\textnormal{Out}}

\newcommand\til{\widetilde}
\newcommand\Tri{\Delta}
\newcommand\SL{\textnormal{SL}}
\newcommand\1{{\bf 1}}

\newcommand{\norm}[1]{\left|#1\right|}

\title{Statistics and compression of scl}
\author{Danny Calegari}
\address{University of Chicago \\ Chicago, Ill 60637 USA}
\email{dannyc@math.uchicago.edu}
\author{Joseph Maher}
\address{Department of Mathematics \\ CUNY College of Staten Island \\
Staten Island, NY 10314; and Department of Mathematics, CUNY Graduate
Center; New York, NY 10016}
\email{joseph.maher@csi.cuny.edu}
\date{version 0.41; February 12, 2013}

\begin{abstract}
We obtain sharp estimates on the growth rate of stable commutator length on random (geodesic)
words, and on random walks, in hyperbolic groups and groups acting nondegenerately on hyperbolic spaces.
In either case, we show that with high probability stable commutator length of an element of
length $n$ is of order $n/\log{n}$. 

This establishes quantitative refinements of qualitative results of Bestvina-Fujiwara and
others on the infinite dimensionality of 2-dimensional bounded cohomology in groups acting
suitably on hyperbolic spaces, in the sense that we
can control the geometry of the unit balls in these normed vector spaces (or rather, in random
subspaces of their normed duals).

As a corollary of our methods, we show that an element obtained by random walk of
length $n$ in a mapping class group cannot be written as a product of fewer than 
$O(n/\log{n})$ reducible elements, with probability going to $1$ as
$n$ goes to infinity. We also show that the translation length on the
complex of free factors of a random walk of length $n$ on the outer 
automorphism group of a free group grows linearly in $n$.
\end{abstract}

\maketitle
\tableofcontents

\section{Introduction}

The goal of this paper is to use probabilistic and ergodic theoretic
methods to obtain bounds and estimates for the growth rate of
characteristic norms in certain classes of groups, especially {\em stable commutator length}
in {\em hyperbolic groups} (and some groups acting on hyperbolic spaces). These growth rates are
a {\em quantitative} obstruction to the existence of nontrivial homomorphisms between groups,
which refine known qualitative obstructions due to Bestvina--Fujiwara \cite{Bestvina_Fujiwara}
and others. They also illustrate a {\em probabilistic}
connection between the growth rate of geometric and algebraic quantities, which
are a manifestation of the phenomenon of Mostow Rigidity in a broad context.

\subsection{Concentration and hyperbolicity}

Let $\xi_n$ be a sequence of random real variables. For many natural sequences $\xi_n$
one can prove a {\em law of large numbers}, i.e.\/ the almost sure existence of a limit
$L:=\lim_{n \to \infty} \xi_n/n$ (this is colloquially known as {\em convergence to the
mean}). Such laws of large numbers hold under remarkably broad hypotheses.

It is natural to look for examples in geometric group theory,
and it is not hard to find them. If $G$ is a group and $\xi:G \to \R^+$ any
subadditive function, one can take $\xi_n$ to be the value of $\xi$
on some Markov process on $G$. For any group $G$, random walk is the
most obvious and important example of a natural Markov process, but
for specific classes of groups, other natural Markov processes exist, 
arising (for example) from an automatic structure on $G$. The law of large
numbers then follows from an application of Kingman's subadditive ergodic theorem.

For example, when $G$ is hyperbolic, Cannon \cite{Cannon} 
showed how to construct a stationary finite Markov chain which gives
a geodesic {\em combing} of $G$. And when $G$ is the fundamental group
of a closed, negatively curved manifold, Ratner \cite{Ratner} showed how to build a
Markov coding for the geodesic flow, whose closed orbits are naturally
in bijection with conjugacy classes in $G$. Roughly speaking, the
{\em geometric} hyperbolicity implies {\em dynamical} hyperbolicity for the
geodesic flow, and this dynamical hyperbolicity gives rise to laws of large numbers.

It is rarer to find examples where $\xi_n$ has {\em sublinear} growth,
but nevertheless concentrates at some deterministic scale.
Colloquially, we say that a sequence of non-negative random variables
$\xi_n$ is {\em concentrated} if there is some deterministic function $f(n)$
so that $\xi_n/f(n)$ converges in probability to a Dirac mass at $1$, and
we say that $\xi_n$ is {\em compressed} if (again for some $f(n)$) every
weak limit of $\xi_n/f(n)$ is a probability measure on $\R^+$
with support bounded away from $0$ and $\infty$.

\subsection{Bounded cohomology}

Bounded cohomology, as introduced by Gromov \cite{Gromov_bounded}, is
(among other things) a functor from the category of groups and 
homomorphisms to the category of normed vector spaces and norm-decreasing
linear maps. One of the main virtues of this functor is its
{\em monotonicity}: if the invariants associated to a group $G$ are 
``smaller'' than the invariants associated to a group $H$, there are
no interesting homomorphisms from $G$ to $H$. As a well-known
example, Bestvina--Fujiwara \cite{Bestvina_Fujiwara} used $2$-dimensional
bounded cohomology to show that every homomorphism from a higher rank
lattice to a mapping class group factors through a finite group (this
fact was known earlier by work of Farb--Masur \cite{Farb_Masur},
building on work of Kaimanovich--Masur \cite{Kaimanovich_Masur}).

To this date, such tools have generally been used somewhat crudely,
because of the enormous difficulty in computing bounded cohomology,
or deriving useful invariants from it. Most authors have concentrated
on bounded cohomology in dimension 2, and have focused almost
exclusively on a trichotomous distinction: namely for a given group $G$,
whether $H^2_b(G)$ (i.e.\/ 2-dimensional bounded cohomology with
real coefficients) is trivial, nontrivial but finite dimensional,
or infinite dimensional. 

In a way, this misses the main point, which
is that $H^2_b(G)$ is canonically a {\em Banach space}. Almost
nothing is known about the (large scale) geometry of this Banach space 
in any nontrivial cases. In this paper, we are able to derive strong geometric
information about the geometry of these Banach spaces in the important
case of hyperbolic groups (and some groups acting suitably on hyperbolic
spaces). This is done via the relationship between 2-dimensional
bounded cohomology, quasimorphisms, and stable commutator length.

For any group $G$ there is an exact sequence of real vector spaces
$$0 \to H^1(G) \to Q(G) \to H^2_b(G) \to H^2(G)$$
where $Q$ denotes the space of {\em homogeneous quasimorphisms} on $G$
(see \S~\ref{scl_background_section} for a precise definition, 
and e.g.\/ \cite{Calegari_scl} Thm.~2.50 for a proof). For a finitely
presented group, $H^1$ and $H^2$ are finite dimensional, so $H^2_b$
and $Q$ carry (almost) the same information. Moreover, $H^2_b$
and $Q/H^1$ are Banach spaces in a functorial way, and the map 
$Q/H^1 \to H^2_b$ is $2$-bilipschitz. We focus on $Q/H^1$ in this paper,
although it would be straightforward to reinterpret our results in
terms of $H^2_b$.

Bavard \cite{Bavard} interpreted $Q/H^1$ in terms of an {\it a priori}
algebraic invariant called {\em stable commutator length} (hereafter
scl). If $G$ is a group and $[G,G]$ denotes its commutator subgroup,
the {\em commutator length} of any $g\in [G,G]$ (denoted $\cl(g)$)
is the least number of commutators in $G$ whose product is equal to $g$,
and the {\em stable commutator length} of $g$ is the limit
$$\scl(g):=\lim_{n \to \infty} \cl(g^n)/n$$
In fact, $\scl$ extends in a natural way to a (pseudo)-norm on the
vector space $B_1(G)$ of real group $1$-boundaries, and descends
to the quotient $B_1^H(G):=B_1(G)/\langle g^n-ng, g-hgh^{-1}\rangle$.
In many important cases, $\scl$ is a {\em norm} on $B_1^H$. In every
case it turns out that $Q/H^1$ is the {\em dual} of $B_1^H$ with its
$\scl$ (pseudo)-norm, where $Q/H^1$ carries the so-called {\em defect
norm} $2D(\cdot)$ (see \S~\ref{scl_background_section}). This statement
is usually known as Generalized Bavard Duality, and in this generality is
proved in \cite{Calegari_scl}.

\subsection{stable commutator length as a random variable}

In this paper we study the growth rate of $\scl$ and the geometry of
$B_1^H$ on {\em random elements} and in {\em random subspaces} of
hyperbolic (and other) groups, with respect to two natural families of
probability measures. First, we obtain results for {\em random
  geodesics} of length $n$ in hyperbolic groups. Secondly, we obtain
results for {\em random walks} of length $n$ in hyperbolic groups, and
in groups acting suitably on hyperbolic spaces. These two senses of
random are conceptually related, but in general become mutually
singular as $n \to \infty$.

There are technical subtleties in either case which are somewhat
complementary. The theory of stable commutator length is
well-adapted to the geometry of quasi-geodesics in hyperbolic groups, since quasi-geodesity
can be certified --- and stable commutator length estimated --- from {\em local} contributions.
A combing determines a geodesic representative of each element in a group,
and the Markov process associated to an automatic structure allows one to pick 
a random element of prescribed word length. But there is no guarantee that the Markov
process in question is {\em ergodic}. Complementarily, the Markov process defining a random
walk is always ergodic, but random paths in hyperbolic spaces are typically not quasi-geodesic. 
Nevertheless, in either context we are able to obtain compression results for $\scl$,
at the deterministic growth rate of $n/\log{n}$.

\subsection{The geometry of $Q$ as an obstruction}\label{Q_obstruction_subsection}

If $G \to H$ is a surjective homomorphism, the pullback $Q(H) \to Q(G)$ is
injective. Hence if $Q(G)$ vanishes but $Q(H)$ does not, no such
surjective homomorphism exists. More generally, no such surjective
homomorphism exists if $Q(G)$ is finite dimensional
but $Q(H)$ is infinite dimensional.

In their seminal paper \cite{Bestvina_Fujiwara}, Bestvina--Fujiwara showed that
if $H$ is a nonelementary subgroup of a mapping class group, then
$Q(H)$ is infinite dimensional. On the other hand, $Q(\Gamma)$ is finite
dimensional whenever $\Gamma$ is a lattice in a higher rank Lie group; consequently
every homomorphism from such a lattice to a mapping class group factors
through an elementary subgroup, and therefore (by Margulis) through a
finite group. Independently Hamenst\"adt \cite{Hamenstadt}
and Bestvina--Feighn \cite{Bestvina_Feighn} proved a similar theorem for
subgroups of $\Out(F_n)$ (Hamenst\"adt's construction considers bounded cohomology
with twisted coefficients).

This application of bounded cohomology makes use of $Q$ as an
$\R$-vector space, but completely neglects its (Banach) geometry.  Any
homomorphism $G \to H$ is nonincreasing for $\scl$. A random walk on
$G$ pushes forward under a homomorphism to a random walk on $H$; the
growth rate of $\scl$ under random walks is therefore an obstruction
to the existence of such a homomorphism. To apply these ideas in
practice one must be able to understand the growth rate of $\scl$
under random walks in the target group with respect to a measure whose
support can be arbitrary; when $H$ is the target group one must be
prepared to consider the case that the support of the measure
generates a subgroup which is not quasiconvex. Carrying this out for
hyperbolic groups, and groups acting in a suitable way on hyperbolic
spaces, is one of the main goals of this paper.

\subsection{Random dynamical systems in the continuum limit}

In 2-dimensions, Pesin theory says that all the entropy of a 
diffeomorphism is carried by Smale horseshoes, which is to say, by 
periodic orbits with hyperbolic dynamics. Up to ``zero entropy'' therefore,
one aims to recover the essential dynamics of a diffeomorphism from an inverse
system of combinatorial data, namely the braid type of the diffeomorphism
modulo a finite invariant subset.

Ghys \cite{Ghys_knots} (building on work of Arnold, Sullivan, Tresser, 
Fathi, Gambaudo and others)
has proposed to understand the group of area-preserving diffeomorphisms
of a surface as a kind of ``limit'' of braid groups of finer and finer meshes 
of discrete points in the surface, and has explained how to use quasimorphisms
on such braid groups to construct similar quasimorphisms on the transformation
groups. This gives highly nontrivial information about the algebraic structure
of such groups, but it is difficult in practice to interpret this information
dynamically. 

One can approach this question from the other direction, by taking families of
random walks in the braid group with respect to finer and finer meshes of discrete 
particles, and trying to obtain a law on random measure-preserving
transformations in the limit. In order to do this, one needs to obtain a
uniform modulus of continuity on the motion of the discrete particles on every
fixed scale, to ensure that the limiting transformations are {\em continuous}.
One way to obtain such estimates would be to show that a random element of a braid
or mapping class group has a short factorization as a product of elements with
support in subsurfaces with simple topology --- i.e.\/ to given an upper bound
on its {\em reducible length}. In fact, our methods give {\em lower bounds}
on reducible length of the same order of magnitude as (stable) commutator length.
It therefore becomes a provocative question whether there are complementary
upper bounds of the same order of magnitude.

\subsection{Summary of results}

We now summarize the remainder of the paper. \S~\ref{scl_background_section}
contains a brief review of the theory of stable commutator length and
quasimorphisms that we use in the sequel. This material is standard.

In \S~\ref{hyperbolic_groups_background_section} we discuss the
ergodic theory of geodesic combings on hyperbolic groups, and establish the
main technical results necessary to obtain estimates on $\scl$ of random
geodesics. Roughly speaking, a combing for a hyperbolic group determines
a directed graph with edges labeled by generators, so that directed paths of
length $n$ in the graph correspond bijectively to elements of the group of
word length $n$ (with respect to some fixed generating set). The directed
graph can be thought of as a topological Markov chain; assigning transition
probabilities to edges determines a stationary Markov process (in the usual sense).
There is an assignment of transition probabilities of maximal entropy corresponding
to the Patterson--Sullivan measure on the group; this allows us to define what
is meant by a ``random element of $G$ of length $n$'' in such a way that we can
produce such elements by a Markov process. 

The first main technical result in
this section is Proposition~\ref{Chernoff_estimate}. In words, this Proposition
gives a Chernoff-type estimate for the number of times any given path
$\sigma$ of length $\ell \log{n}/\log{\lambda}$ occurs in a random path of 
length $n$ in an ergodic Markov process, 
where $\ell<1$ is fixed, and $\lambda$ is the entropy of the 
Markov process. The key subtlety is that the length of the subpaths we are
counting is not fixed, but depends on the length of the big random path.
The second main technical result in this section is the content of
Lemmas~\ref{mu_i_measures_roughly_equal} and \ref{mu_approximate_inverse} 
which says that the subwords of fixed length appearing in
different ergodic components of the Patterson--Sullivan Markov process are
``coarsely equivalent'' (as isometry classes of geodesic
segments in the group) in distribution.

In \S~\ref{random_geodesic_section} we apply the estimates obtained in
\S~\ref{hyperbolic_groups_background_section} to prove concentration
for $\scl$ on random geodesics in hyperbolic groups at the scale
$n/\log{n}$:

\begin{hyperbolic_thm}
Let $G$ be a hyperbolic group, and $S$ a finite generating set. There are
constants $\con{1}$, $\con{2}>0$, $\con{3}>0$, $\con{4}>1$, $\con{5}>0$ so that if $g$ is a random element
with $|g|_S \in [n-\con{1},n+\con{1}]$ conditioned to lie in $[G,G]$, then 
$$\Pr(\con{2}n/\log{n} \le \scl(g) \le \con{3}n/\log{n})=1-O(\con{4}^{-n^{\con{5}}})$$
In fact, we obtain the stronger result $\cl(g)\le \con{3}n/\log{n}$, with the same
estimate in probability.
\end{hyperbolic_thm}

The proofs of the upper and lower bounds are somewhat different. The proof of the upper bound is
related directly to the definition of (stable) commutator length, 
and depends on writing a random $g$ efficiently as a product of commutators by 
directly pairing almost inverse subwords of $g$ of length $O(\log{n})$, using the
Chernoff-type estimate in Proposition~\ref{Chernoff_estimate}. The proof of the lower bound is
obtained via Bavard duality, by constructing a homogeneous quasimorphism with controlled defect
whose value on $g$ is large. The homogeneous quasimorphism is a variant on the
{\em small counting quasimorphisms} considered by Epstein--Fujiwara \cite{Epstein_Fujiwara},
for which, in contrast with the ``big'' counting quasimorphisms considered by Brooks
\cite{Brooks}, controlling the defect is easy. On the other hand, while obtaining the correct
order of magnitude estimate of the defect is straightforward, obtaining exact estimates is
substantially harder; this explains the gap between the bounds $\con{2}n/\log{n}$, $\con{3}n/\log{n}$
in the Hyperbolic Geodesic Theorem, which we believe to be an artifact of the method of
proof, rather than a genuine statistical phenomenon. In fact, the authors of \cite{Calegari_Walker}
conjectured that one should be able to take $\con{2}$ and $\con{3}$ arbitrarily close to
$\log{\lambda}/6$, where $\lambda$ is the growth entropy of $G$ with respect to $S$ 
(i.e.\/ the number such that the ball of radius $n$ in $G$ in the word metric
with respect to $S$ has $\Theta(\lambda^n)$ elements).

In \S~\ref{random_walk_section} we obtain {\em universal} estimates on
$\scl$ on random walks in finitely generated groups, and obtain
precise order of magnitude estimates for hyperbolic groups and certain
groups acting on hyperbolic spaces. For simplicity we consider random
walks obtained by repeatedly convolving a symmetric measure of finite
support, although our methods could presumably be extended to a more
general context. Because of the symmetry of the Cayley graph of a free
group with respect to a free generating set, our sharp estimates for
$\scl$ on random geodesics in free groups gives sharp estimates for
$\scl$ on random walks in free groups; together with the monotonicity
of $\scl$ under homomorphisms, this gives the following universal
upper bound:

\begin{upper_bound_thm}
Let $G$ be a group with a finite symmetric generating set $S$, and let $|S|=2k$. 
Let $g$ be obtained
by random walk on $G$ (with respect to $S$) of length $n$ (even), conditioned to lie in $[G,G]$.
Then for any $\epsilon>0$ there are constants $\con{1}>1$, $\con{2}>0$ so that with
probability $1-O(\con{1}^{-n^{\con{2}}})$ there is an inequality
$$\scl(g) \le ((k-1)\log(2k-1)+\epsilon)/6k\cdot n/\log{n}$$
\end{upper_bound_thm}

The remainder of the section is devoted to finding lower bounds of the same order of magnitude
for groups acting in a suitable way on hyperbolic spaces. We consider a group $G$
acting simplicially on a $\delta$-hyperbolic simplicial complex $Y$, and let $\mu$ be a symmetric
probability measure with support equal to a finite subset $S$ of $G$ 
(note that we do not assume that $S$ generates $G$). We say that the triple $(G,\mu,Y)$ is
{\em nondegenerate} if it satisfies a short list of axioms, spelled out in
Definition~\ref{def:nondegenerate}. These axioms are satisfied in the following important
cases:
\begin{enumerate}
\item{$G$ is a hyperbolic group, $Y$ is the Cayley graph of $G$, and $\langle S \rangle$ is
a nonelementary subgroup;}
\item{$G$ is hyperbolic relative to a family of subgroups, $Y$ is the Groves--Manning
hyperbolic complex certifying relative hyperbolicity (see \cite{Groves_Manning}, \S~3), 
and $\langle S \rangle$ is nonelementary;}
\item{$G$ is the mapping class group of a surface, $Y$ is the complex of curves, and $\langle S\rangle$
is not reducible or virtually abelian.}
\end{enumerate}
A discussion of more general examples (including the action of $\Out(F_n)$ on
certain hyperbolic complexes) is deferred to \S~\ref{subsection:outfn}.

In these terms, the main theorem proved in this section is the Hyperbolic Lower Bound Theorem:

\begin{hyperbolic_lb_thm}
Let $Y$ be a $\delta$-hyperbolic simplicial complex (not assumed to be locally finite),
and let $G$ be a finitely generated group that
acts simplicially on $Y$. Let $\mu$ be a symmetric probability measure of
finite support on $G$ so that $(G,\mu,Y)$ is nondegenerate, in the sense of 
Definition~\ref{def:nondegenerate}.

Let $g$ be obtained by random walk on $G$ (with respect to $\mu$) of length $n$. Then
for any $\con{1}$ there is a $\con{2}>0$ and $\con{3}$ depending only on $\delta$ and $G$,
so that with probability at least $1-n^{-\con{1}}$ there is a homogeneous quasimorphism
$\phi$ on $G$ satisfying the following properties:
\begin{enumerate}
\item{$\phi(g)\ge n\con{2}/\log{n}$;}
\item{$D(\phi)\le \con{3}$;}
\item{$|\phi(h)|\le 2d_Y(q,hq)\con{3}/\log{n}$ for any $g \in G$ and any $q\in Y$.}
\end{enumerate}
In particular, for any $\con{1}>0$ there is a constant $C>1$ 
so that if we condition on $g \in [G,G]$ (for $n$ even), then
$$\Pr(C^{-1}n/\log{n} \le \scl(g) \le Cn/\log{n}) \ge 1-n^{-\con{1}}$$
\end{hyperbolic_lb_thm}

Two significant corollaries follow immediately:

\begin{reducible_cor}
Let $\mu$ be a symmetric probability measure on the mapping class group of $\Sigma$
of finite support, and suppose the subgroup it generates is not reducible or virtually
abelian. Then for any $\con{1}$ there is $C$ so that if $g$ is obtained by random walk
in $G$ of length $n$, then $g$ cannot be expressed as a product of fewer than
$Cn/\log{n}$ reducible elements, with probability at least $1-n^{-\con{1}}$.
\end{reducible_cor}

\begin{growth_obstruction_cor}
Let $G$ be a group and $\mu$ a symmetric probability measure with finite support which
generates $G$. Suppose that for any $\epsilon>0$ there is a $\delta>0$ so that if $g$ is
obtained by random walk on $G$ of length $n$ (even) conditioned to lie in $[G,G]$, we have
$$\Pr(\scl(g)<\epsilon\cdot n/\log{n})>\delta$$
Then every homomorphism from $G$ to a hyperbolic group or to a mapping class group has
virtually abelian image.
\end{growth_obstruction_cor}

The Growth Obstruction Corollary is the promised quantitative strengthening of 
\cite{Bestvina_Fujiwara} alluded to in \S~\ref{Q_obstruction_subsection}.

In \S~\ref{random_norm_subsection} we spell out the analogs of these theorems for 
parametric families of elements, in terms of the geometry of the unit ball in the
$\scl$ norm of a random subspace of $B_1^H(G)$ for $G$ as above. In words: we show that the
unit ball in a random subspace has uniformly bounded geometry (up to rescaling
by a deterministic factor), and conjecture that it is $C^0$ close to a cross-polytope (i.e.\/
to a unit ball in an $L^1$ norm). This conjecture is only known to be true for free groups,
by \cite{Calegari_Walker}.

In \S~\ref{subsection:outfn} we consider random walks on non-locally compact Gromov
hyperbolic spaces in some generality, and prove (under mild conditions)
that all the axioms from Definition~\ref{def:nondegenerate} are satisfied except possibly the
condition of acylindricity. A key intermediate step is to show
that random walks converge to the (Gromov) boundary almost surely.
Although this does not have immediate implications
for $\scl$, it enables us to obtain linear lower bounds on translation length
of elements obtained by random walk. By ``translation length'' $\tau(g)$ of an
isometry $g$ on a metric space $Y$ we mean here {\em asymptotic translation length};
i.e.\/  $\lim_{n \to \infty} d_Y(y,g^n(y))/n$ for arbitrary $y \in Y$.
Precisely, we show:

\begin{linear_translation_thm}
Let $G$ be a group of isometries of a Gromov hyperbolic space $Y$,
which is not necessarily locally compact. Let $\mu$ be a probability
distribution with finite support on $G$, such that the support of
$\mu$ generates a non-elementary subgroup of $G$. Then
there are constants $L > 0$ and $c < 1$ such that
\[ \Pr( \tau(w_n) \le Ln ) \le O(c^n), \]
where $w_n$ is the group element obtained by a random walk of length
$n$, and $\tau(w_n)$ is the translation length of $w_n$ acting on $Y$.
\end{linear_translation_thm}

This theorem applies (for example) to the action of $\Out(F_n)$ on
the free splitting complex, and on the complex of free factors, both of which are
(Gromov) hyperbolic; see Bestvina-Feighn \cite{Bestvina_Feighn, Bestvina_Feighn_free_factor} and 
Handel-Mosher \cite{Handel_Mosher} for details.

Finally in \S~\ref{lower_bound_section} we survey what is known for
arbitrary groups, and deduce universal lower bounds on $\scl$ (known to
be sharp in certain cases) from the work of Bj\"orklund--Hartnick
\cite{Bjorklund_Hartnick}. Burger--Monod \cite{Burger_Monod} showed
that any quasimorphism differs by a bounded amount from a {\em
  harmonic function}. In particular, the expectation of this
(adjusted) quasimorphism on a random walk is constant; i.e.\/ the
value of this function is a martingale. By applying the martingale
central limit theorem, Bj\"orklund--Hartnick were able to obtain a
central limit theorem for the distribution of values of a
quasimorphism under a random walk. Their result applies in great
generality, but if one specializes to finitely generated groups $G$
with $Q(G)$ finite dimensional and $H_1(G)$ torsion for simplicity,
Bavard duality plus the main theorem of \cite{Bjorklund_Hartnick}
implies that for any $\epsilon$ there are positive constants $a,b$
(depending only on $\epsilon$) so that $\Pr(a < \scl_n/\sqrt{n} <
b)\ge 1-\epsilon$, where $\scl_n$ denotes the value of $\scl$ under a
random walk of length $n$. Colloquially one could say that $\scl$ has
growth rate of order $\sqrt{n}$ in this case. We obtained a special
case of the theorem of Bj\"orklund--Hartnick before their work was
announced; because the method of proof is more geometric, we decided
it was worth including in \S~\ref{circle_clt_subsection}. One nice
geometric corollary we obtain is as follows:

\begin{area_winding_thm}
Fix some angle $\alpha$ and length $\ell$. Let $P_n$ be a random polygon in the hyperbolic plane
with (cyclic) vertices $p_0,p_1,\cdots,p_n$, where $d(p_i,p_{i+1})=\ell$ for each $0\le i\le n-1$
and an angle of $\pm \alpha$ at each $p_i$ with $0<i<n$, with signs independently and
uniformly chosen from $\pm 1$. Let $A_n$ be the algebraic area enclosed by $P_n$,
and $W_n$ the winding number of $\partial P_n$. Then
$A_n$ and $W_n$ both satisfy a central limit theorem with mean $0$.
\end{area_winding_thm}

\section{Stable commutator length}\label{scl_background_section}

We recall some standard definitions and facts for the convenience of the reader.
A basic reference for the material in this section is \cite{Calegari_scl}.

\subsection{Stable commutator length}

\begin{definition}
Let $G$ be a group, and $G'$ its commutator subgroup. Given $g \in G'$, the {\em commutator
length} of $g$, denoted $\cl(g)$, is the least number of commutators in $G$ whose product is
$g$, and the {\em stable commutator length}, denoted $\scl(g)$, is the limit
$$\scl(g) = \lim_{n \to \infty} \frac {\cl(g^n)} {n}$$
\end{definition}

The following estimates are elementary:

\begin{lemma}\label{product_estimate}
Let $G$ be a group, and $g$, $h$ elements of $G'$. Then there is an inequality
$$\scl(gh) \le \scl(g) + \scl(h) + 1/2$$
\end{lemma}

\begin{lemma}\label{length_bounds_scl}
Let $G$ be a group with finite symmetric generating set $S$. Let $g \in G'$. Then there
is a constant $\con{1}$ depending on $G$ and $S$ so that
$$\scl(g) \le \con{1}|g|_S$$
where $|\cdot|_S$ denotes word length with respect to $S$.
\end{lemma}

\subsection{Quasimorphisms}

There is a duality between stable commutator length and certain functions on $G$ 
called {\em homogeneous quasimorphisms}.

\begin{definition}
Let $G$ be a group. A function $\phi:G \to \R$ is a {\em quasimorphism} if there is some 
least non-negative number $D(\phi)$ called the {\em defect}, so that for all $g,h \in G$, there is an inequality
$$|\phi(gh) - \phi(g) - \phi(h)| \le D(\phi)$$
A quasimorphism is {\em homogeneous} if, further, it satisfies $\phi(g^n)=n\phi(g)$ for all
$g \in G$ and all $n \in \Z$.
\end{definition}

Denote the vector space of all quasimorphisms on $G$ by $\widehat{Q}(G)$, and the subspace of homogeneous
quasimorphisms by $Q(G)$. 

\begin{lemma}[\cite{Calegari_scl}, Lem.~2.21 and 2.58.]\label{homogenization_is_quasimorphism}
Given any $\psi \in \widehat{Q}$, the {\em homogenization} $\overline{\psi}$,
defined by
$$\overline{\psi}(g) = \lim_{n \to \infty} \psi(g^n)/n$$
exists and satisfies $|\overline{\psi} -\psi| \le D(\psi)$. 
Moreover, $\overline{\psi}$ is a homogeneous quasimorphism with 
$D(\overline{\psi}) \le 2D(\psi)$.
\end{lemma}

Bavard duality is the following theorem:

\begin{theorem}[Bavard \cite{Bavard}; \cite{Calegari_scl} Thm.~2.70]\label{Bavard_duality_theorem}
If $G$ is a group and $g\in G'$ then
$$\scl(g) = \sup_{\phi \in Q(G)} \frac {\phi(g)} {2D(\phi)}$$
\end{theorem}

We use Bavard duality to obtain lower bounds on stable commutator length.

\section{Hyperbolic groups}\label{hyperbolic_groups_background_section}

A standard introduction to the theory of hyperbolic groups is \cite{Gromov_hyperbolic}. More
specialized references for the material in this section are 
\cite{Cannon}, \cite{Coornaert} and \cite{Calegari_Fujiwara}.

\subsection{Hyperbolic groups}

Let $G$ be a group with a generating set $S$. Let $C_S(G)$ (or just $C$ for brevity) denote the
Cayley graph of $G$ with respect to the generating set $S$. The Cayley graph $C_S(G)$ can be
made into a geodesic metric space, by giving each edge length $1$. With this metric,
$d(\id,g) = |g|_S$.

\begin{definition}
A group $G$ is {\em hyperbolic} if $C_S(G)$ is hyperbolic as a metric space. That is,
if there is some $\delta$ so that if $pqr$ is a geodesic triangle, any point on the
geodesic $pq$ is within distance $\delta$ of $qr \cup rp$.
\end{definition}

We are casual about identifying $G$ with the vertices of $C_S(G)$, and in this way
think of $G$ as a metric space. If $S$ is given, we say $G$ is {\em $\delta$-hyperbolic}
for any $\delta$ as above.

There is an equivalence relation on proper geodesic rays in $C_S(G)$, where two
rays are equivalent iff they are a finite Hausdorff distance apart. The set of
equivalence classes is itself a compact space called the {\em Gromov boundary} of $G$,
and denoted $\partial_\infty G$. The left action of $G$ on itself (or on $C_S(G)$)
extends to an action on $\partial_\infty G$ by homeomorphisms.

\subsection{Combings}

For an introduction to combings, regular languages etc. see \cite{Epstein_et_al}.

Let $S^*$ denote the set of finite words in the generating set $S$, let $|\cdot|$
denote word length in $S^*$, and let $\eval:S^* \to G$ denote the evaluation map. 
A word $w\in S^*$ is a {\em geodesic} if $|\eval(w)|_S = |w|$. Under evaluation of
prefixes, a word $w\in S^*$ determines a {\em directed path} in the Cayley graph $C_S(G)$ from
$\id$ to $\eval(w)$. We denote this path by $\path(w)$.

Suppose $X$ is a finite directed
graph (hereafter digraph) with a distinguished initial vertex, and with edges
labeled by elements of $S$ in such a way that there is at most one outgoing edge
from each vertex with a given label. 

Let $\Gamma$ denote the set of finite directed simplicial paths in $X$, and
$\Gamma_0$ the subset starting at 
the initial vertex. There is an injective map $\word:\Gamma_0 \to S^*$ which takes a
path $\gamma$ to the string of edge labels on the edges it traverses, in order.
The composition $\path\circ\word$ takes paths in $\Gamma_0$ to paths in $C_S(G)$ starting
at $\id$. When the meaning is clear from context, we denote $\path\circ\word(\gamma)$ simply
by $\gamma$.

A subset of the form $\word(\Gamma_0)\subset S^*$ is necessarily prefix-closed,
since $\Gamma_0$ is closed under taking initial subpaths. Furthermore, $\word(\Gamma_0)$
is a regular language; in fact, a subset $L\subset S^*$ is a prefix-closed regular
language if and only if there is some $X$ with $\word(\Gamma_0)=L$.

\begin{definition}\label{combing_definition}
A {\em combing} of $G$ (with respect to a generating set $S$)
is a subset $L \subset S^*$ for which there is some labeled
digraph $X$ as above with $\word(\Gamma_0)=L$ such that
\begin{enumerate}
\item{the evaluation map $\eval:L \to G$ is a bijection; and}
\item{the words of $L$ are geodesics.}
\end{enumerate}
We say the digraph $X$ {\em parameterizes} the combing.
\end{definition}

Note that $L$ is prefix-closed by our conventions. Note also that for a combing, the image
of $\gamma \in \Gamma_0$ under $\path\circ\word$ is a geodesic in $C_S(G)$.

\begin{remark}
Many conflicting definitions of combings appear in the literature. Our definition is
by no means standard.
\end{remark}

\begin{theorem}[Cannon \cite{Cannon}]
Let $G$ be a hyperbolic group, and let $S$ be a generating set. Then there is
a combing for $G$ with respect to $S$.
\end{theorem}

\begin{example}
An ordering of $S$ determines a lexicographic order (i.e.\/ dictionary order)
on $S^*$. In any hyperbolic group $G$, the language $L$ of {\em lexicographically
first geodesic representatives} is an example of a combing in the sense of
Definition~\ref{combing_definition}.
\end{example}

\subsection{Markov chains}

The digraph $X$ is a topological Markov chain. Associated to $X$ is the {\em transition
matrix} $M$ whose $M_{ij}$ entry counts the number of directed edges from vertex $i$
to vertex $j$. Note that $M$ has a non-negative real eigenvalue $\lambda$ do that
$|\xi|\le \lambda$ for every other eigenvalue $\xi$ of $M$.

Two vertices in a digraph are said to be {\em communicating} if there are
directed paths between the vertices in either direction. This defines an equivalence
relation on $X$, and we define a {\em component} to be a maximal subgraph whose vertices
are a communicating class. There is a natural digraph $C(X)$ obtained as a quotient
of $X$, whose vertices are the communicating classes of $X$. From the definition it
is clear that the digraph $C(X)$ has no directed loops.

If $C$ is a component of $X$, the adjacency matrix $M_C$ of $C$ has biggest real
eigenvalue $\lambda_C\le \lambda$. 

\begin{definition}
A component $C$ is {\em maximal} if $\lambda_C = \lambda$.
\end{definition}

In general, there are no constraints on $C(X)$ other than the fact that it has no
directed loops. However, for $G$ a hyperbolic group, and $X$ a digraph parameterizing
a combing, we have the following, due (implicitly) to Coornaert:

\begin{theorem}[Coornaert \cite{Coornaert}; \cite{Calegari_Fujiwara} Lem.~4.15]
Let $G$ be a hyperbolic group, and let $X$ be a digraph parameterizing a combing.
Then each directed path in $C(X)$ contains at most one maximal component.
Equivalently, if $\lambda$ is the maximal real eigenvalue of $M$, there are positive
constants $\con{1}$, $\con{2}$ so that
$$\con{1}\lambda^n \le |G_n| \le \con{2}\lambda^n$$
\end{theorem}

\begin{remark}
Actually, Coornaert does not use the language of combings, and
only proves the {\it a priori} weaker fact that
$C_1\lambda^n \le |G_{\le n}| \le C_2 \lambda^n$. The stronger fact with
$G_{\le n}$ replaced by $G_n$ follows immediately once one realizes that $G_n$
counts the number of elements of length $n$ in a prefix-closed regular language.
As far as we know, this sharper observation appears for the first time in print 
in \cite{Calegari_Fujiwara}.
\end{remark}

There is a natural stationary Markov
chain with states the vertices of $X$, for which the
whose positive transition probabilities are a subset of the (directed) edges of $X$.
This Markov chain is described in detail in \cite{Calegari_Fujiwara}, \S~4. 
We let $N_{ij}$ denote the probability of a transition from state $i$ to state $j$
(note $N_{ij}=0$ if $M_{ij}=0$) and $\mu$ a certain stationary measure for $N$ of maximal
entropy.

\begin{lemma}[\cite{Calegari_Fujiwara} Lem.~4.9--10]
$N$ is a stochastic matrix, and $\mu$ is stationary for $N$. The support of $\mu$
is the union of the maximal components of $X$.
\end{lemma}

For each $n$, let $\Gamma^n$ denote the subset of $\Gamma$ consisting of paths of
length $n$, and let $\Gamma^n_0$ denote the subset starting at the initial vertex. Note that
$\Gamma^n_0$ bijects with $G_n$ under $\eval\circ\word$.
There is a probability measure on $\Gamma^n$ which we denote $\mu$
by abuse of notation, defined by
$$\mu(\gamma) = \mu(\gamma(0))N_{\gamma(0)\gamma(1)}N_{\gamma(1)\gamma(2)}\cdots N_{\gamma(n-1)\gamma(n)}$$
Thinking of $\Gamma^n$ as the cylinder sets in the space $\Gamma^\infty$ of infinite
paths, there is an associated probability measure on $\Gamma^\infty$ which is
invariant under the shift map $s$ which takes a path to the suffix obtained by omitting
the first vertex.

We can also define a measure $\nu$ on $\Gamma^n_0$ by
$$\nu(\gamma) = N_{0,\gamma(1)}N_{\gamma(1)\gamma(2)}\cdots N_{\gamma(n-1)\gamma(n)}$$
(since $\gamma(0)=0$ by definition for $\gamma \in \Gamma^n_0$), and extend to a measure $\nu$
on $\Gamma^\infty_0$. The shift map $s$ takes $\Gamma^\infty_0$ into $\Gamma^\infty$, and we
have $\mu = \lim_{n \to \infty} \frac 1 n \sum_{i=0}^{n-1} s^i_*\nu$. Note that
$\nu$ is just $\mu$ conditioned on $\Gamma_0 \subset \Gamma$ (or on $\Gamma_0^n \subset \Gamma^n$ for
each $n$). Note too that if the maximal components of $X$ (the support of $\mu$) are all
aperiodic, then $\mu = \lim_{n \to \infty} s^n_*\nu$, and in general $\mu = \lim_{n \to \infty} \frac
1 m \sum_{i=0}^{m-1} s^{i+n}_*\nu$ where $m$ is the lcm of the periods of the maximal components.

\subsection{Ergodicity at infinity}

For each $n$, let $\nu_n$ be the probability measure on $G$ defined by
$$\nu_n = \frac {\sum_{|g|_S\le n} \lambda^{-|g|_S}\delta_g} {\sum_{|g|_S\le n} \lambda^{-|g|_S}}$$
where $\delta_g$ is the Dirac measure supported at $g$.
This extends trivially to a probability measure on $G\cup \partial_\infty G$. It turns
out that the limit $\nu=\lim_{n \to \infty} \nu_n$ exists and is supported on
$\partial_\infty G$, where it is known as a {\em Patterson-Sullivan} measure.

\begin{theorem}[Coornaert \cite{Coornaert}, Thm.~7.7]\label{ergodic_at_infinity}
Let $\nu$ be the Patterson-Sullivan measure. The action of $G$ on $\partial_\infty G$
preserves the measure class of $\nu$, and is ergodic.
\end{theorem}

The relationship between the measures $\nu_n$ and the measures $\nu$ on $\Gamma^n_0$
is as follows. Let $\gamma \in \Gamma^n_0$ and let $g=\eval(\word(\gamma))$.
Define $\cone(g)$ to be the image in $G$ of the union of the $\gamma' \in \Gamma^m_0$
extending $\gamma$. Then $\nu(\gamma) = \nu(\cone(g)):=\lim_{n \to \infty} \nu_n(\cone(g))$.
Let $\gamma' \in \Gamma^{n+1}_0$ be obtained by extending $\gamma$, and let $h=\eval(\word(\gamma'))$
so that $h=gs$ for some generator $s$. Let $i$ be the terminal vertex of $\gamma$, and $j$
the terminal vertex of $\gamma'$. Then $N_{ij} = \nu(\cone(h))/\nu(\cone(g))$.
See \cite{Calegari_Fujiwara} \S~4.3. for details.

The bijection $\Gamma^n_0 \to G_n$ lets us compare the $\nu$-measure on $\Gamma^n_0$ with the
uniform measure on $G_n$. From the relationship between $\nu$ and $N_{ij}$, 
the $\nu$-measure of $g \in G_n$ (i.e.\/ $\nu(\cone(g))$) 
depends only on cone type of $g$, or equivalently on
the terminal vertex of the corresponding path in $\Gamma^n_0$. Since there are only finitely many
cone types, and since every $g\in G$ is a uniformly bounded distance from some $g'$
whose cone has growth $\Theta(\lambda^n)$ (and therefore has $\nu(g')$ bounded away from
$0$) we derive the following proposition:

\begin{proposition}[\cite{Calegari_Fujiwara}, \S~4]\label{measures_quasi_equivalent}
There are positive constants $\con{1},\con{2},\con{3}$ with the following property.
Let $g\in G_n$, and let $B$ denote the intersection of the ball of radius $\con{1}$
about $g$ (in $C_S(G)$) with $G_n$. Let $B_\Gamma$ denote the preimage of $B$ under the natural
bijection $\Gamma^n_0 \to G_n$. Then there is an inequality
$$\con{2}\nu(B_\Gamma)/\nu(\Gamma^n_0) \le |B|/|G_n| \le \con{3}\nu(B_\Gamma)/\nu(\Gamma^n_0)$$
\end{proposition}

Finally, we state a proposition which describes the typical behavior of a $\nu$-random element
of $\Gamma^n_0$.

\begin{proposition}[\cite{Calegari_Fujiwara}, Lem.~4.10]\label{path_in_one_component}
Let $\gamma$ be a $\nu$-random element of $\Gamma^n_0$. Then as $n \to \infty$,
with probability $1-O(\con{1}^{-\con{2}n})$, apart from a prefix of size $\con{3}\log{n}$, 
the path $\gamma$ is entirely contained in a single maximal component $C$, and
it enters (and stays in) a given maximal component $C$ with probability $\mu(C)$.
\end{proposition}

The collective significance of these propositions can be summarized in the following way:
every $g\in G_n$ is a uniformly bounded distance away from some $h \in G_n$ corresponding to
a path $\gamma \in \Gamma^n_0$ which is generated (apart from a prefix of length $O(\log{n})$)
by one of finitely many ergodic stationary Markov processes, corresponding to one of the
maximal components of $X$.

We will establish properties which hold for typical sequences in each of these stationary
Markov processes. Sequences in different Markov processes can be compared using 
Theorem~\ref{ergodic_at_infinity}. This will let us establish properties which hold for
$\nu$-typical $g\in G_n$, and we will deduce that the properties also hold for
typical $g\in G_n$ (with respect to the uniform measure).

\subsection{Mixing times and a Chernoff-type estimate}

Let $X$ be a finite stationary ergodic Markov chain with transition matrix $N$, 
and $\mu$ the stationary measure. Define measures $\mu$ on $\Gamma^n$ and $\Gamma^\infty$
as before. In our applications, $X$ will be a maximal component $C(X)$ of $X$ as above, and $\mu$
will be $\mu$ conditioned on $C(X)$.

If $\sigma,\gamma \in \Gamma$, define $C_\sigma(\gamma)$ to be the number of times $\sigma$
appears as a subsequence of $\gamma$. Now if $\gamma_n$ is a $\mu$-random element of $\Gamma_n$,
then ergodicity of $X$ implies that
$\frac 1 n C_\sigma(\gamma_n) \to \mu(\sigma)$ in probability.
A {\em Chernoff-type estimate} says that the probability that the deviation 
$|C_\sigma(\gamma_n)-n\mu(\sigma)|$ is of size $\Theta(n^{\delta})$ decays exponentially
in $n$, for $\delta$ in a suitable range (the critical exponent is half the exponent of $n\mu(\sigma)$,
corresponding to the distribution in the central limit theorem). 
For our purposes it is necessary to estimate the probability when $|\sigma|=O(\log{n})$.

Let $\lambda$ be the entropy of $X$ (as above), and let $m=\log_\lambda n = \log{n}/\log{\lambda}$
so that there are $O(n^\ell)$ elements of $\Gamma$ of length $\ell m$.

\begin{proposition}\label{Chernoff_estimate}
Suppose $\sigma$ has length $\ell m$ with $\ell < 1$. Then for any $\epsilon>0$ there
are constants $\con{1}>1$ and $\con{2}>0$ (depending on $X$ and $\ell$ 
but not on $n$ or $\sigma$) so that
$$\Pr\left(|C_\sigma(\gamma_n) - n\mu(\sigma)| < n^{\epsilon + (1-\ell)/2}\right) = 1 - O(\con{1}^{-n^{\con{2}}})$$
In fact, there is a fixed integer $L$ so that for each residue $i$ mod $Lm$, if $\gamma_{i,k}$ is the
subpath of $\gamma$ of length $\ell m$ starting at $i+kLm$, and $C_{\sigma,i}(\gamma_n)$ is the
number of copies of $\sigma$ amongst the $\gamma_{i,k}$, then
$$\Pr\left(|C_{\sigma,i}(\gamma_n) - \frac {n} {Lm} \mu(\sigma)| < n^{\epsilon+(1-\ell)/2}\right) = 1-O(\con{1}^{-n^{\con{2}}})$$
\end{proposition}
\begin{proof}
This proposition is essentially a special case of a theorem of Dinwoodie \cite{Dinwoodie}, and 
we sketch the proof, referring to \cite{Dinwoodie} for details 
(also compare \cite{Calegari_Walker} \S~2). For simplicity, we assume $X$ is aperiodic (the general
case is very similar).

In fact, the first estimate follows from the second, so we prove the second.
Choose some big fixed integer $L$, and let $i$ be a residue mod $Lm$. Let $\gamma_{i,k}$ be the
subpath of $\gamma$ of length $\ell m$ starting at $i+kLm$. For fixed $i$, 
the successive $\gamma_{i,k}$ are not independent, but their correlation is extremely small.
Explicitly, if $\lambda_1<1$ is the second largest eigenvalue of the transition matrix $N$, then
conditional on any given value of $\gamma_{i,k}$, the probability that the 
first vertex of $\gamma_{i,k+1}$ is equal to some $j$ differs from $\mu(j)$ by 
$\lambda_1^{(L-\ell)m}$. So conditioned on any given value of
$\gamma_{i,k}$, the probability that $\gamma_{i,k+1}$
is equal to $\sigma$ is bounded above by $\mu(\sigma)(1+\lambda_1^{(L-\ell)m}/\mu(\sigma(0)))$ and below by
$\mu(\sigma)(1-\lambda_1^{(L-\ell)m}/\mu(\sigma(0)))$ where $\sigma(0)$ is the initial vertex of $\sigma$. If
$L$ is sufficiently large, these bounds are approximately $\mu(\sigma)(1\pm n^{-\con{3}})$
where $\con{3}$ may be taken to be as large as we like.

Let $C_{\sigma,i}(\gamma_n)$ be the number of copies of $\sigma$ amongst the $\gamma_{i,k}$. By
the argument above, $C_{\sigma,i}(\gamma_n)$ may be estimated from the Chernoff bound for 
{\em independent} Bernoulli variables; see e.g.\/ \cite{Stroock}, Thm.~1.3.13. 
Explicitly, we obtain an estimate of the form
$$\Pr\left(|C_{\sigma,i}(\gamma_n) - \frac {n} {Lm} \mu(\sigma)(1\pm n^{-\con{3}})| < n^{\epsilon+(1-\ell)/2}\right) = 1-O(\con{1}^{-n^{\con{2}}})$$
Taking $\con{3}$ sufficiently large, we may absorb it into the $n^{\epsilon + (1-\ell)/2}$ term, at
the cost of increasing $\epsilon$ an arbitrarily small amount.

Summing over $i$ we get
$$\Pr\left(|C_{\sigma}(\gamma_n) - \mu(\sigma)| < Lm\cdot n^{\epsilon+(1-\ell)/2}\right) = 1 - O(Lm\cdot \con{1}^{-n^{\con{2}}})$$
Since $Lm = O(\log{n})$, by adjusting constants the result holds.
\end{proof}

\begin{remark}
It is worth spelling out the meaning of $\lambda$, $\Gamma^n$, $\mu$ and so on for a simple
example. For example, if $F$ is the free group of rank $k$, after fixing a free generating
set, there is a combing of $F$ consisting of the regular language of all reduced words in the generators.
Then $\Gamma^n$ consists of the set of reduced words of length $n$, the entropy
$\lambda$ is equal to $2k-1$ since $|\Gamma^n|=2k\cdot(2k-1)^{n-1}$, and $\mu$ is the uniform
probability measure on $\Gamma^n$ for each $n$.
\end{remark}

\begin{remark}
The meaning of Proposition~\ref{Chernoff_estimate} is somewhat hard to glean directly from the
statement. It can be explained informally as follows. We would like to prove a strong equidistribution theorem
for the set of subpaths of $\gamma$ of length $\ell m$. If these subpaths were independent,
the Chernoff bound would give us what we want. But the subpaths are not independent --- in
fact, two subpaths with a big overlap are highly dependent on each other. However, the correlations
between subpaths decay {\em exponentially quickly}, and therefore a collection of subpaths with
the property that the distance (in $\gamma$) between distinct subpaths is at least $Lm$ for some
fixed $L \gg 1$ are ``independent enough'' for a modified Chernoff bound to apply. The set of
all subpaths of length $\ell m$ can be partitioned into $Lm$ subsets so that each subset satisfies
the desired separation property, and we get an equidistribution result that holds on each
subset with high probability. Combining the subsets together has the effect merely of multiplying
the error by $Lm = O(\log(n))$ which can be absorbed into one of the constants.
\end{remark}

\subsection{Anti-alignment}

\begin{definition}
Two oriented geodesic segments $\gamma_1,\gamma_2$ in $C_S(G)$ are {\em $K$ anti-aligned} if there
is $g\in G$ so that the terminal point of $g(\gamma_1)$ is within distance $K$ of the initial point
of $\gamma_2$, and vice versa.
\end{definition}

By the defining property of $\delta$-hyperbolicity, the translate of $\gamma_1$ is contained
in the $K+2\delta$-neighborhood of the translate of $\gamma_2$, and vice versa. Note the
property of being $K$ anti-aligned is a property of {\em orbit classes} of geodesic segments
in $C_S(G)$ under the action of $G$. Given any path $\gamma \in \Gamma$, the orbit class of
$\gamma$ in $C_S(G)$ is well-defined, so it makes sense to ask if a pair of elements of $\Gamma$
are $K$ anti-aligned.

\begin{proposition}\label{not_many_anti_aligned}
Fix $L>2$ and some constant $K$. 
Let $\gamma$ be a $\nu$-random element of $\Gamma^n_0$, and let $\gamma_i$ denote
the successive subpaths of $\gamma$ of length $Lm$.
Then there is some $\epsilon>0$ and constant $\con{1},\con{2}$ so that
$$\Pr\left(\#\lbrace i:\gamma_i \text{ $K$ anti-aligns some } \gamma' \subset \gamma \rbrace < n^{1-\epsilon}\right)
= 1 - O(\con{1}^{-n^{\con{2}}})$$
\end{proposition}
\begin{proof}
By Proposition~\ref{path_in_one_component}, we can assume that apart from a prefix of size $O(\log{n})$,
$\gamma$ is contained in some maximal component $C$. 

Now, if $\gamma_i,\gamma'$ subpaths of length $Lm$ are $K$ anti-aligned, then there are
{\em disjoint} $\gamma_i' \subset \gamma_i$ and $\gamma'' \subset \gamma'$ of length at least $Lm/2$
which are $K+2\delta$ anti-aligned. Replace $K$ by $K+2\delta$ and $L$ by $L/2$, and by abuse of
notation let $\gamma_i$ denote successive subpaths of length $Lm$. Then we need only show that
$$\Pr\left(\#\lbrace i:\gamma_i \text{ $K$ anti-aligns some } \gamma' \text{ disjoint from } \gamma_i \rbrace < n^{1-\epsilon}\right)
= 1 - O(\con{1}^{-n^{\con{2}}})$$

For any given location of $\gamma'$ in $\gamma$, there are $O(n^L)$ distinct elements of
$\Gamma^n_0$ which agree with $\gamma$ outside $\gamma'$, all with approximately the same
$\mu$-measure. If $\gamma'$ anti-aligns some $\gamma_i$, then $\gamma'$ is determined by
$\gamma_i$, by $h_1,h_2\in G$ of length $\le K$, and by the initial vertex of $\gamma'$ in $X$.
Hence for each $\gamma_i$ there are only $\con{3}$ possible choices for $\gamma'$, and
therefore only $O(n)$ choices where we range over all $i$. In particular, conditioning on
$\gamma$ outside $\gamma'$, the probability that a given $\gamma'$ $K$ anti-aligns some $\gamma_i$
that it is disjoint from is at most $O(n^{1-L})$. The proposition follows.
\end{proof}

\subsection{Comparing components}

Proposition~\ref{Chernoff_estimate} lets us very accurately estimate the number
of copies of a subword $\sigma$ in a random word $\gamma_n$ of length $n$
in a maximal component $C$ of $X$, at least when $|\sigma|=\ell m$ with $\ell<1$
for $m=\log_\lambda n$.

However, it is necessary to compare such distributions for {\em different}
maximal components. We introduce some notation.

For each maximal component $C$, let $\mu|_C$ be the conditional
probability measure with support in $C$. That is, define $\mu(\cdot)|_C = \mu(\cdot)/\mu(C)$ in $C$, and $\mu(\cdot)|_C=0$ on the complement
of $C$. By abuse of notation, this defines a probability measure $\mu|_C$ on $\Gamma$,
defined on cylinder sets by
$$\mu|_C(\gamma)=\mu|_C(\gamma(0))N_{\gamma(0)\gamma(1)}\cdots N_{\gamma(n-1)\gamma(n)}$$
Proposition~\ref{Chernoff_estimate} says that for $\gamma_n$ a $\nu$-random path
in $X$ conditioned to enter a specific component $C$, we have 
$|C_\sigma(\gamma_n)-\mu|_C(\sigma)|<n^{\epsilon+(1-\ell)/2}$ with exponentially
few exceptions. Under the evaluation map, for each $M$
the measure $\mu|_C$ on $\Gamma^M$ pushes forward
to a probability measure $\mu|_C$ on $G_M$. We would like to compare the
measures $\mu|_C$ on $G_M$ for different maximal components $C$.

It is too much to hope that these measures will be {\em equal}. However, it turns out
that each $\mu|_{C_i}$ can be obtained from $\mu|_{C_j}$ by (roughly speaking)
a random convolution process, as described in the following lemma:

\begin{lemma}\label{mu_i_measures_roughly_equal}
Suppose the Cayley graph $C_S(G)$ is $\delta$-hyperbolic. Then the following is true:
for each pair of big components $C_i$, $C_j$ of $X$, and for each $M$,
there is a map $f_{i,j}:G_M \to \text{Prob}(G_M)$ satisfying
\begin{enumerate}
\item{for all $g$ and all $h$ in the support of $f_{i,j}(g)$ there is $a\in G$ with $|a|\le \delta$ 
so that $d(g,aha^{-1})\le 2\delta$; and}
\item{$\mu|_{C_j}(h) = \sum_g \mu|_{C_i}(g)f_{i,j}(g)(h)$.}
\end{enumerate}
We express bullet~(2) by saying that $\mu|_{C_j}$ is obtained by {\em convolving} 
$\mu|_{C_i}$ with $f_{i,j}$.
\end{lemma}
\begin{proof}
In fact, the proof of this lemma is a very simple
trick, which is a variation on the main trick of \cite{Calegari_Fujiwara}.
By Kakutani's random ergodic theorem (see \cite{Kakutani}; alternately, this follows
from Proposition~\ref{Chernoff_estimate}) 
and the definition of $\mu_{i,m}$, 
almost every infinite path in $X$ that enters $C_i$
is composed of subpaths of length $M$ that are distributed in $G_M$ according to $\mu|_{C_i}$.
Call an infinite path (i.e.\/ an element of $\Gamma$) 
{\em $C_i$-typical} if its subpaths have this property.
Let $\gamma_i$ be a $C_i$-typical infinite path, and by abuse of notation, let $\gamma_i$
denote the corresponding infinite geodesic path in $G$ starting at $\id$, limiting to
some point in $\partial_\infty G$. By ergodicity of the action of $G$ on $\partial_\infty G$ 
(i.e.\ Theorem~\ref{ergodic_at_infinity}), there are typical $\gamma_i$ and $\gamma_j$ and some
$a\in G$ so that $a\cdot \gamma_i$ and $\gamma_j$ have the same endpoint, and therefore outside
some compact subset, they synchronously fellow-travel. 
It follows that we can subdivide $a\cdot\gamma_i$
and $\gamma_j$ into subpaths of length $M$, throwing away finitely many at the start, so that corresponding
subpaths are each contained in the $\delta$-neighborhood of each other. If two associated
subpaths evaluate to $g$ and $h$ in $G_M$ respectively, then $g=a_1ha_2$ where each of $a_1,a_2$ has length at most $\delta$.
Hence $d(g,a_1ha_1^{-1})\le 2\delta$. So define $f_{i,j}(g)(h)$ to be the probability that a subpath
of $\gamma_i$ evaluating to $g$ fellow-travels (as above) a subpath of $\gamma_j$ evaluating to $h$.
\end{proof}

We conclude that each individual measure $\mu|_{C_i}$ can be convolved (in the sense above)
by a weighted average $\sum p_jf_{i,j}$ to obtain a single probability measure
$\mu'$ on $G_M$. Now, though it might not be true
that $\mu'(v)=\mu'(v^{-1})$ for all $v \in G_M$, this is {\em approximately} 
true, up to convolution in the sense of Lemma~\ref{mu_i_measures_roughly_equal}:

\begin{lemma}\label{mu_approximate_inverse}
Suppose the Cayley graph $C_S(G)$ is $\delta$-hyperbolic. Then the following is true:
for each $M$ there is a map $f:G_M \to \text{Prob}(G_M)$ satisfying
\begin{enumerate}
\item{for all $g$ and all $h$ is the support of $f(g)$ there is $a\in G$ with $|a|\le \delta$ 
so that $d(g,ah^{-1}a^{-1})\le 2\delta$; and}
\item{$\mu'(h) = \sum_g \mu'(g)f(g)(h)$.}
\end{enumerate}
\end{lemma}
\begin{proof}
Choose some $N\gg M$, so that with probability $1-O(\con{1}^{-n^{\con{2}}})$, a random
path $\gamma$ in $\Gamma^N$ satisfies $|C_\sigma(\gamma)-N\mu|_{C_i}(\sigma)|<N^{1-\con{3}}$
for every $\sigma\in \Gamma^M$, where $C_i$ is the big component that $\gamma$ enters
(the existence of such an $N$ given $M$ follows from Proposition~\ref{Chernoff_estimate}).

We call a $\gamma$ for which such an estimate holds (for some $i$) {\em almost typical}.
By Proposition~\ref{measures_quasi_equivalent}, there are almost typical paths
$\gamma$ and $\gamma'$ evaluating to $g$ and $g'$ in $G_N$ respectively, and satisfying
$d(g^{-1},g')\le \con{4}$.

Let $\gamma^{-1}$
denote the path in $G$ from $\id$ to $g^{-1}$ obtained by reversing and translating
$\gamma$ (note that $\gamma^{-1}$
will not typically be an element of $\Gamma$). Since subpaths of $\gamma^{-1}$ are in bijection with
subpaths of $\gamma$ but oriented oppositely, $C_\sigma(\gamma)=C_{\sigma^{-1}}(\gamma^{-1})$
for all $\sigma$ of length $M$.
Since $\gamma^{-1}$ and $\gamma'$ synchronously fellow-travel, after throwing away suffixes of
uniformly bounded length (depending only on $\con{4}$), 
we can pair subpaths of $\gamma^{-1}$ and $\gamma'$ of length $M$ so that 
corresponding subpaths are each contained in the $\delta$-neighborhood of each other. 
Taking $N \to \infty$ (for fixed $M$) the proof follows.
\end{proof}

\section{$\scl$ of random geodesics in hyperbolic groups}\label{random_geodesic_section}

We now have nearly all the necessary tools to estimate $\scl$ on random elements in hyperbolic
groups. There is one significant additional complication for hyperbolic groups $G$ for which
$H_1(G)$ has positive rank, namely that $\scl$ is only defined on elements in the commutator
subgroup.

We will show in \S~\ref{homology_section} that the relative proportion of $[G,G]$
in the set of elements of length $[n-\con{1},n+\con{1}]$ (for a suitable constant $\con{1}$)
is of size $O(n^{-k/2})$ where $k$ is the rank of $H_1(G)$; in particular, it is
{\em polynomial} in $n$. It follows that properties of words that hold with probability
$1-O(\con{2}^{-n^{\con{3}}})$ in $G_n$ will also hold with a similar estimate in probability
for words conditioned to lie in $[G,G]$.

\subsection{Almost pairing subwords}

\begin{proposition}\label{cl_upper_bound}
Let $G$ be a hyperbolic group, and let $S$ be a generating set. Let $\lambda$ be the
growth rate for $G$ (i.e.\/ the number such that $|G_n|=\Theta(\lambda^n)$). Then for any
$\epsilon>0$, there are constants
$\con{1}>1$, $\con{2}>0$ so that if $g$ is a random element of $G_n$, with probability
$1-O(\con{1}^{-n^{\con{2}}})$ the following holds:
\begin{enumerate}
\item{there is some $h$ with $|h|_S\le 8\delta \cdot n(\log{\lambda}+\epsilon)/\log{n}+o(n/\log{n})$; and}
\item{the product $gh^{-1}$ is in $[G,G]$; and}
\item{there is an estimate $\cl(gh^{-1})\le n(\log{\lambda}+\epsilon)/2\log{n}$.}
\end{enumerate}
\end{proposition}
\begin{proof}
The proof follows by assembling the results of the previous section. 
Let $\gamma \in \Gamma_0^n$ correspond to $g$. By Proposition~\ref{measures_quasi_equivalent},
it suffices to let $\gamma$ be a $\nu$-random element of $\Gamma_0^n$.
With the desired probability, we can assume
apart from a prefix and suffix of size $O(\log{n})$, that $\gamma$ is contained in a single
maximal component $C$. Fix some $\ell < 1$, and consider the set of successive subpaths $\gamma_i$
of $\gamma$ of length $\ell m$, where $m=\log{n}/\log{\lambda}$. By Proposition~\ref{Chernoff_estimate},
the distribution of the $\gamma_i$ is very close to $\mu|_C$. 
By Lemma~\ref{mu_i_measures_roughly_equal} and Lemma~\ref{mu_approximate_inverse} we can
pair most of the $\gamma_i$ in such a way that the evaluation of each pair is approximately 
inverse; i.e.\/ if $g_i,g_j$ are the values of the pair of segments in $G$, 
there are $a_1,a_2$ with $|a_i|\le 4\delta$ and $g_i=a_1g_j^{-1}a_2$. 
We can therefore cancel each $g_i$ with $g_j$ at the cost of a commutator plus a word of
length at most $8\delta$. The product of these $8\delta$ words is $h$. 
All but $o(n/\log{n})$ can be paired in this way; the remainder can be absorbed into $h$
at the cost of adjusting $\epsilon$ an arbitrarily small amount (for big $n$). The proof follows.
\end{proof}

\subsection{Counting quasimorphisms}\label{counting_quasimorphism_subsection}

We obtain lower bounds on $\scl$ via Bavard duality, by constructing explicit quasimorphisms with
uniformly bounded defect, and value $O(n/\log{n})$ on random $g$. The quasimorphisms
in question are a variant on the {\em small counting quasimorphisms} of
Epstein-Fujiwara \cite{Epstein_Fujiwara}. In fact, it is useful to work in the generality
of a group acting on a $\delta$-hyperbolic complex, following \cite{Fujiwara} (also
see \cite{Calegari_scl}, \S~3.5).

Let $Y$ be a $\delta$-hyperbolic simplicial complex ({\em not} necessarily locally
finite) and let $G$ act on $Y$ simplicially.
If $\sigma$ is a finite oriented simplicial path in $Y$,
let $\sigma^{-1}$ denote the same path with the opposite orientation.
A {\em copy} of $\sigma$ is a translate $g\cdot\sigma$ for some $g\in G$.

\begin{definition}
Let $Y$ be a $\delta$-hyperbolic simplicial complex, and let $p \in Y$ be a base point.
Let $\Sigma$ be a (possibly infinite) collection of oriented simplicial paths in $Y$,
and let $\Sigma^{-1}$ denote the collection obtained by reversing the orientations on
all $\sigma \in \Sigma$. For any oriented simplicial path $\gamma$ in $Y$ define 
$$c_\Sigma(\gamma)=\text{maximal number of {\em disjoint} copies of }\sigma \in \Sigma\text{ contained in }\gamma$$
and then for $g\in G$ define
$$c_\Sigma(g)=d(p,g(p)) - \inf_\gamma(\length(\gamma)-c_\Sigma(\gamma))$$
where the infimum is taken over {\em all} oriented simplicial paths $\gamma$ from $p$
to $g(p)$. Define the {\em small counting quasimorphism} $h_\Sigma$ by the formula
$$h_\Sigma(g):=c_\Sigma(g) - c_{\Sigma^{-1}}(g)$$
\end{definition}

A path $\gamma$ from $p$ to $g(p)$
realizing the infimum of $\length(\gamma)-c_\Sigma(\gamma)$ is called a {\em realizing
path} for $\gamma$. Since the values of this function are integers, a realizing path
exists.

\begin{lemma}[Fujiwara \cite{Fujiwara}, Lem.~3.3]\label{Fujiwara_antialign_lemma}
Suppose the length of every $\sigma \in \Sigma$ is $\ge 2$. Then any realizing 
path is a $(2,4)$-quasigeodesic.
\end{lemma}
It follows that any realizing path is within distance $\con{1}$ of a geodesic,
where $\con{1}$ depends only on $\delta$.

\begin{lemma}[Fujiwara \cite{Fujiwara}, Prop.~3.10]\label{Fujiwara_defect_lemma}
Suppose the length of every $\sigma \in \Sigma$ is $\ge 2$. Then there is a constant
$\con{2}$ depending only on $\delta$ such that $D(h_\Sigma)\le \con{2}$.
\end{lemma}

\begin{remark}
Actually, Epstein-Fujiwara only consider small counting functions for a single
$\sigma$, but the bounds on the geometry of realizing paths and on the defect of
$D(h_\Sigma)$ are valid for arbitrary collections $\Sigma$ as above. This is by
contrast with the ``big'' counting functions introduced by Brooks \cite{Brooks},
in which one counts {\em all} (possibly overlapping)
copies of $\sigma$, not just a maximal disjoint collection.
\end{remark}

\begin{remark}
As defined, the quasimorphism $c_\Sigma$ depends on the choice of 
basepoint $p$. However, different choices of points $p$ give rise to quasimorphisms
with the same homogenization. Since it is the homogenization we really care about,
we gloss over this detail.
\end{remark}

We apply this construction to the case that $Y=C_S(G)$ and $p=\id$ so that
$d(p,g(p))=|g|_S$.

\begin{proposition}\label{quasimorphism_lower_bound}
Let $G$ be a hyperbolic group, and let $S$ be a generating set. Let $\lambda$
be the growth rate for $G$. Then for any $\epsilon>0$ there are constants
$\con{1}>1,\con{2}>0$ and $\con{3}$ depending only on $\delta$
so that if $g$ is a random element of $G_n$, with
probability $1-O(\con{1}^{-n^{\con{2}}})$, there is a homogeneous
quasimorphism $\phi$ on $G$ satisfying the following properties:
\begin{enumerate}
\item{$\phi(g)\ge n\log{\lambda}/(2+\epsilon)\log{n}$;}
\item{$D(\phi)\le\con{3}$;}
\item{$|\phi(h)|\le 2|h|_S\log{\lambda}/(2-\epsilon)\log{n}$ for any $h\in G$.}
\end{enumerate}
\end{proposition}
\begin{proof}
By Lemma~\ref{measures_quasi_equivalent} 
it suffices to prove the theorem for a $\nu$-random element $g$.
Fix $L=2+\epsilon$, and recall the notation $m=\log{n}/\log{\lambda}$.
Let $\gamma \in \Gamma_0^n$ correspond to $g$, and let $\gamma_i$ denote the
successive subpaths of $\gamma$ of length $Lm$. For a suitable constant $K$ (to be 
determined shortly), let $\Sigma$ be the subset of the $\gamma_i$ 
which do not $K$ anti-align any $\gamma' \subset \gamma$. If $c_\Sigma$ denotes
the small counting function associated to the set $\Sigma$ (thought of as a collection
of orbit classes of subpaths of the Cayley graph $C_S(G)$) then
by Proposition~\ref{not_many_anti_aligned} we obtain an estimate
$c_\Sigma(g)\ge n\log{\lambda}/(2+\epsilon)\log{n}$.

Now, there is a constant $K$ so that any $(2,4)$-quasigeodesic in a $\delta$-hyperbolic space 
stays within distance $K$ of a genuine geodesic. If we choose $K$ with this
property, then by Lemma~\ref{Fujiwara_antialign_lemma} and the definition of
$\Sigma$ we have $c_{\Sigma^{-1}}(g)=0$, and therefore
$h_{\Sigma}(g)\ge n\log{\lambda}/(2+\epsilon)\log{n}$.
On the other hand, $D(\phi)\le\con{3}$ for some constant $\con{3}$
depending only on $\delta$, by Lemma~\ref{Fujiwara_defect_lemma}.

Finally, if $h$ is any element of word length $|h|_S$, any realizing path for
$c_\Sigma$ or $c_{\Sigma^{-1}}$ is a $(2,4)$-quasigeodesic, which therefore
has length at most $2|h|_S+4$ and therefore
contains at most $(2|h|_S+4)/Lm$ disjoint paths of length $Lm$. 
\end{proof}

In particular, if $g \in [G,G]$ then $\scl(g)\ge \con{1}n/\log{n}$ with very
high probability, and if there is $h$ with $|h|_S=O(n/\log{n})$ and
$gh^{-1}\in [G,G]$ then $\scl(gh^{-1})\ge \con{2}n/\log{n}$.

\subsection{Homology}\label{homology_section}

Proposition~\ref{cl_upper_bound} and Proposition~\ref{quasimorphism_lower_bound} 
together give upper and lower bounds
on $\scl(gh^{-1})$ for some $h$ with $|h|_S=O(n/\log{n})$ and $gh^{-1}\in[G,G]$. 
If $g$ is in $[G,G]$ then so is $h$, and $\scl(h)$ can be estimated from its length
(see Lemma~\ref{homologically_trivial_bound}). We will see in this section
that the relative proportion of homologically trivial $g$ in $G_n$ is polynomial
of bounded degree, and therefore a random element of $G_n$ {\em conditioned to
lie in $[G,G]$} will have two-sided bounds on $\scl$, with high probability.

The following lemma is elementary:

\begin{lemma}\label{homologically_trivial_bound}
Let $G$ be a group and $S$ a finite generating set. There is a constant $\con{1}$
so that if $h$ is an element of $[G,G]$ then $\cl(h) \le \con{1}|h|_S$.
\end{lemma}
\begin{proof}
For simplicity, we suppose $H_1(G)$ is torsion-free; the general case is not
significantly harder.

Assume without loss of generality that $S=S^{-1}$, and
let $L$ be the free abelian semigroup generated by $S$ as a set. We think
of $L$ as the intersection of the positive orthant in $\R^{|S|}$ with $\Z^{|S|}$.
The map from $S$ to $H_1(G;\R)$ extends by linearity to $\R^{|S|}$, and the
kernel $K$ is a rational subspace. It follows that $K\cap L$ is finitely
generated as an abelian semigroup (such a generating set is known as a
{\em Hilbert basis}); see e.g.\/ \cite{Barvinok}. It follows that if we write
$h$ as a (geodesic) word in the elements of $S$, there is a constant $\con{1}$
so that there is a subset of letters of cardinality at most $\con{1}$ whose
image in homology is trivial. We move these letters to the right, at the cost
of at most $\con{1}$ commutators. In other words, we can write $h=h_1h_2h_3$
where each $h_i$ is in $[G,G]$, where $|h_1|_S < |h|_S$, where $|h_2|_S \le \con{1}$,
and where $\cl(h_3) \le \con{1}$. Since there are only finitely many elements
of $G$ with $|\cdot|_S\le \con{1}$, there is a bound on the commutator length of
the homologically trivial ones. Hence $|h_2|_S\le \con{2}$ and therefore by
induction, $\cl(h) \le (\con{1}+\con{2})\cdot|h|_S$ as required.
\end{proof}

The main proposition of this section estimates the relative density of $[G,G]$
in $G_n$.

\begin{proposition}\label{homology_local_limit_theorem}
Let $G$ be a hyperbolic group with finite generating set $S$, and suppose the
rank of $H_1(G)$ is $k$. There are constants $\con{1}$, $\con{2}$ so that
if $g$ is a random element of $[n-\con{1},n+\con{1}]$ (with the uniform probability),
then $\Pr(g \in [G,G])\ge \con{2}\cdot n^{-k/2}$. 
\end{proposition}
\begin{proof}
The case that $G$ is free and $S$ is a free generating set is due to
Sharp \cite{Sharp}, with $\con{1}=1$. However, the proof does not use very particular
properties of free groups, and generalizes to hyperbolic groups. The key probabilistic
component is a (standard) local limit theorem for random sums in ergodic (finite,
stationary) Markov chains.

For general hyperbolic groups, one cannot apply such local limit theorems directly
because $X$ might have more than one maximal component; however, we can apply
local limit theorems to each maximal component {\em individually}. To do this
we need to know that the {\em expected} value in $H_1(G)$ of (the evaluation of)
a random walk conditioned to lie in a maximal 
component $C_i$ of $X$ is zero. This follows from the main theorem
of Calegari-Fujiwara \cite{Calegari_Fujiwara}, since a homomorphism to $\Z$ 
is an example of a bicombable function.
It follows that the Markov chains associated to each component $C_i$ satisfy
the conditions in \cite{Sharp}, Thm.~2 and we can obtain sharp estimates of the
desired form on the probability that the evaluation of a random path in $C_i$ 
has trivial abelianization.

Now, a $\nu$-random $\gamma\in \Gamma_0^n$ is of the form $\gamma_1\gamma_2$
where $\gamma_2$ is contained in some maximal component. It is not true that
we can bound the length of $\gamma_1$, but it {\em is} true that there is
a constant $\con{1}$ so that $\Pr(|\gamma_1|\le \con{1})\ge 1/2$. If $\alpha$
denotes the abelianization map $\alpha:\Gamma^n \to H_1(G)$, then $|\gamma_1|\le \con{1}$
implies a uniform estimate $|\alpha(\gamma_1)|\le \con{2}$.

Because of the Markov property, we get an estimate for $g$ a
$\nu$-random word of length $n$ of the form
$$\Pr(|\alpha(g)|\le \con{2}) \ge \con{3}\cdot n^{-k/2}$$
By Proposition~\ref{measures_quasi_equivalent}, we get a similar estimate 
(but with different constants) for $g$ a random element of $G_n$ with the 
uniform probability. Now, there is a constant $\con{4}$ so that 
for every $g$ with $|\alpha(g)|\le\con{2}$ there is $h$ with $|h|_S\le \con{2}\con{4}$
so that $gh\in [G,G]$. The map $g \to gh$ is bounded-to-one, so the cardinality
of the intersection of $[G,G]$ with the set of words of length in the interval
$[n-\con{2}\con{4},n+\con{2}\con{4}]$ is at least $\con{5}n^{-k/2}\cdot|G_n|$.
The proof follows.
\end{proof}

Putting this all together, we obtain the main theorem in this section.

\begin{theorem}[Hyperbolic geodesic theorem]\label{hyperbolic_geodesic_theorem}
Let $G$ be a hyperbolic group, and $S$ a finite generating set. There are
constants $\con{1}$, $\con{2}>0$, $\con{3}>0$, $\con{4}>1$, $\con{5}>0$ so that if $g$ is a random element
with $|g|_S \in [n-\con{1},n+\con{1}]$ conditioned to lie in $[G,G]$, then 
$$\Pr(\con{2}n/\log{n} \le \scl(g) \le \con{3}n/\log{n})=1-O(\con{4}^{-n^{\con{5}}})$$
In fact, we obtain the stronger result $\cl(g)\le \con{3}n/\log{n}$, with the same
estimate in probability.
\end{theorem}
\begin{proof}
The estimates in Proposition~\ref{quasimorphism_lower_bound} and
Proposition~\ref{cl_upper_bound} hold with probability $1-O(C^{-n^c})$, and therefore
they still hold with the same order of probability (with different constants) 
conditioned on $g \in [G,G]$, by Proposition~\ref{homology_local_limit_theorem}.
Bavard duality (Theorem~\ref{Bavard_duality_theorem}) therefore gives the lower bound.

Similarly, for random $g$ conditioned to lie in $[G,G]$, Proposition~\ref{cl_upper_bound}
says we can write $g=gh^{-1}h$ where $\cl(gh^{-1})\le Cn/\log{n}$, and
where $|h|_S\le Cn/\log{n}$. Since $g\in[G,G]$ and $gh^{-1}\in[G,G]$, we have
$h\in[G,G]$. So Lemma~\ref{homologically_trivial_bound} gives
$\cl(h)\le Cn/\log{n}$. Putting this together gives the upper bound on $\cl(g)$,
with the desired estimate in probability.
\end{proof}

As remarked in the introduction, the gap between the upper and lower bounds is presumably
an artefact of the method of proof; in fact in \cite{Calegari_Walker} the authors
conjectured that there should be concentration for the random variable $\scl(g)\log{n}/n$
at $\log{\lambda}/6$, where $\lambda$ is the growth entropy of $G$ with respect to the
generating set $S$. The main theorem of \cite{Calegari_Walker} proves this for a free
group with respect to a free generating set (see Theorem~\ref{free_group_theorem} below
for a precise statement).

\section{$\scl$ of random walks in groups}\label{random_walk_section}

In this section we obtain estimates on the value of $\scl(g)$ where $g$ is
obtained by a {\em random walk} in a group $G$, providing $G$ satisfies certain
hypotheses. Even in a hyperbolic group $G$ with a fixed generating set $S$,
the probability distributions defined by random geodesics and by random walks are
not usually uniformly comparable, and typically become mutually singular at infinity.
However, there is one very important special case in which the two probability
distributions can be compared very precisely, namely the case of free groups
with a free generating set. It follows that we obtain upper bounds on $\scl$
in random words in free groups, with high probability.

The significance of this is not that we are interested in free groups {\it per se},
but rather that $\scl$ is {\em monotone} under homomorphisms. If $G$ is any group,
and $S$ any (symmetric) generating set, then there is a surjective homomorphism
$F_S \to G$ where $F_S$ is the free group on $S$, and (simple)
random walk on $F_S$ pushes
forward to random walk on $G$. It follows that any upper bound on $\scl$ on
random walks in free groups gives a universal upper bound on $\scl$ on random walks in
{\em any} group $G$.

These upper bounds are complemented by lower bounds for hyperbolic groups, and
for certain groups acting on hyperbolic spaces (e.g.\/ braid groups, mapping class
groups). Universal lower bounds, valid for all groups, are obtained by a quite
different method, and discussed in the next section.

\subsection{Special case: free groups}

For $F$ a free group of rank $k$, Calegari-Walker obtained a sharpening of
Theorem~\ref{hyperbolic_geodesic_theorem}:

\begin{theorem}[Calegari-Walker \cite{Calegari_Walker}, Thm.~4.1]\label{free_group_theorem}
Let $F$ be a free group of rank $k$, and let $g$ be a random element of
length $n$ in a free generating set where $n$ is even, and $g$ is conditioned
to lie in $[F,F]$. Then for every $\epsilon>0$ and every $C$ there is an estimate
$$\Pr(|\scl(g)\log{n}/n - \log(2k-1)/6|\le \epsilon)=1-O(n^{-C})$$
\end{theorem}

In fact, the upper bound (i.e.\/ $\scl(g)\log{n}/n - \log(2k-1)/6)\le \epsilon$)
is proved to hold with probability $1-O(C^{-n^c})$ for some $C>1,c>0$; see
\cite{Calegari_Walker}, Prop.~4.2. Theorem~\ref{free_group_theorem} is derived
from a proposition, valid for $g$ random of length $n$, and then
conditioning on $g\in [F,F]$. The following proposition is implicit in
\cite{Calegari_Walker}; for completeness, we indicate how it
follows immediately from \S~4.1--3 of that paper. 

\begin{proposition}\label{free_scl_upper_bound}
Let $F$ be a free group of rank $k$. Then for any
$\epsilon>0$, there are constants
$\con{1}>1$, $\con{2}>0$ and $\delta>0$ 
so that if $g$ is a random element of $F_n$, with probability
$1-O(\con{1}^{-n^{\con{2}}})$ the following holds:
\begin{enumerate}
\item{there is some $h$ with $|h|_S\le O(n^{1/2+\epsilon})$; and}
\item{the product $gh^{-1}$ is in $[F,F]$; and}
\item{there is an estimate $\scl(gh^{-1})\le n(\log(2k-1)+\epsilon)/6\log{n}$.}
\end{enumerate}
\end{proposition}
\begin{proof}
It is convenient to use the (well-known) extended definition of $\scl$ as a norm on homologically
trivial formal real group $1$-boundaries; see \cite{Calegari_scl}, \S~2.6 for details.
Lem.~4.7 \cite{Calegari_Walker} says that there is some formal $1$-chain
$\Gamma$ with $|\Gamma|_S = O(n^{1-\delta})$ for some $\delta>0$ so that
$\scl(g-\Gamma)\le n(\log(2k-1)+\epsilon)/6\log{n}$. By an estimate of Rivin
(see e.g.\/ \cite{Rivin}) we can assume that the $L^1$ norm of $[g]\in H_1(F)=\Z^k$
is $O(n^{1/2+\epsilon})$ for any $\epsilon$, with probability $1-O(\con{1}^{-n^{\con{2}}})$, and
therefore there is $h$ with $[g]=[h]$ and $|h|_S\le O(n^{1/2+\epsilon})$. We estimate
\begin{align*}
\scl(gh^{-1}) & \le \scl(g+h^{-1})+1/2 \\
			  & \le \scl(g-\Gamma) + \scl(\Gamma + h^{-1})+1/2  \\
			  & \le n(\log(2k-1)+\epsilon)/6\log{n} + O(n^{1-\delta})
\end{align*}
and the $O(n^{1-\delta})$ may be absorbed into the $\epsilon$.
\end{proof}

\begin{remark}
The reader who is uncomfortable with the proof of Proposition~\ref{free_scl_upper_bound}
can safely use Proposition~\ref{cl_upper_bound} instead in the sequel, after
observing that $\delta=0$ in a free group with a free generating set. 
The only cost is that the constant in Proposition~\ref{cl_upper_bound} is
worse by a factor of $6$, whereas the constant in Proposition~\ref{free_scl_upper_bound}
is sharp.
\end{remark}

\subsection{Universal upper bounds}

We now compare random words with random walks. The
Cayley graph of $F$ with respect to a free generating set is a regular $2k$-valent tree.
The group of simplicial automorphisms of this tree, fixing the origin, acts transitively
on the set of vertices at distance $m$, for any $m$. Let $\mu$ be the uniform probability
measure on the generators, and let $\mu^{*n}$ denote the $n$-fold convolution; i.e.\/ the
probability measure associated to a random walk of length $n$. Then $\mu^{*n}$ is a
weighted sum of uniform measures on the sets $F_m$ for $m\le n$. The generating function
for the weights can be determined explicitly (see e.g.\/ \cite{Woess} Lem.~1.24), and
a straightforward calculation shows that for any $\epsilon$,
all but $O(\con{1}^{-n^{\con{2}}})$ of the mass of $\mu^{*n}$ is concentrated on the
set of $F_m$ with $m/n \in [(k-1)/k-\epsilon,(k-1)/k+\epsilon]$.
We therefore we obtain the following proposition:

\begin{proposition}\label{free_scl_upper_bound_random_walk}
Let $F$ be a free group of rank $k$. Then for any
$\epsilon>0$, there are constants
$\con{1}>1$, $\con{2}>0$ and $\delta>0$ 
so that if $g$ is obtained by random walk on $F$ (in a free generating set) of length $n$, 
with probability
$1-O(\con{1}^{-n^{\con{2}}})$ the following holds:
\begin{enumerate}
\item{there is some $h$ with $|h|_S\le O(n^{1/2+\epsilon})$; and}
\item{the product $gh^{-1}$ is in $[F,F]$; and}
\item{there is an estimate $\scl(gh^{-1})\le ((k-1)\log(2k-1)+\epsilon)/6k\cdot n/\log{n}$.}
\end{enumerate}
\end{proposition}

Remarkably, from this elementary estimate, we obtain a {\em universal sharp} upper bound
on $\scl$ for random walks in {\em arbitrary} finitely generated groups.

\begin{theorem}[Universal upper bound]\label{universal_upper_bound_random_walk}
Let $G$ be a group with a finite symmetric generating set $S$, and let $|S|=2k$. 
Let $g$ be obtained
by random walk on $G$ (with respect to $S$) of length $n$ (even), conditioned to lie in $[G,G]$.
Then for any $\epsilon>0$ there are constants $\con{1}>1$, $\con{2}>0$ so that with
probability $1-O(\con{1}^{-n^{\con{2}}})$ there is an inequality
$$\scl(g) \le ((k-1)\log(2k-1)+\epsilon)/6k\cdot n/\log{n}$$
\end{theorem}
\begin{proof}
Let $\phi:F_k \to G$ take a free symmetric generating set for $F_k$ to $S$. Then random walk
in $F_k$ (with respect to the standard generating set) pushes forward to random walk in $G$
with respect to $S$. Since $\scl$ is monotone nonincreasing under homomorphisms, the
theorem follows from Proposition~\ref{free_scl_upper_bound_random_walk} and
Lemma~\ref{homologically_trivial_bound}, together with the fact that a random walk of length
$n$ has probability at least $\Theta(n^{-k/2})$ of being homologically trivial (for $n$ even).
\end{proof}

\subsection{Template for obtaining lower bounds}

We will obtain lower bounds, complementing Theorem~\ref{universal_upper_bound_random_walk},
for random walks in hyperbolic groups, and certain groups acting on hyperbolic spaces;
the most important example of the latter will be (not necessarily quasiconvex) finitely
generated subgroups of hyperbolic groups. The lower bounds are obtained from
the counting quasimorphism construction, described in \S~\ref{counting_quasimorphism_subsection};
however, the argument is complicated by the fact that a random walk in a hyperbolic group
(or in a hyperbolic graph) is almost certainly not quasigeodesic.

The abstract template for obtaining lower bounds is the following somewhat
technical proposition, which is basically just a restatement of some properties
of small counting quasimorphisms.

\begin{proposition}\label{lower_bound_template}
Let $G$ be a group acting by isometries on
a $\delta$-hyperbolic simplicial complex $Y$ ({\em not} assumed to be locally finite)
with a basepoint $p$. Let $g\in G$ be given, and let $\gamma(g)$ (or $\gamma$ for short)
be a geodesic in $Y$ from $p$ to $g(p)$. We set $n=\length(\gamma)$.
Fix $\con{1}>0$, and let $\gamma_i$
be the successive (nonoverlapping) subpaths of $\gamma$ of length $\con{1}\log{n}$.
Let $K(\delta)$ (or $K$ for short) be such that any $(2,4)$-quasigeodesic in a $\delta$-hyperbolic
space stays within distance $K$ of a genuine geodesic.
Suppose that there is an $\epsilon>0$ for which the following inequality holds:
$$\#\lbrace i:\gamma_i \text{ $K$ anti-aligns some } \gamma' \subset \gamma\rbrace < n^{1-\epsilon}$$
Then there are constants $\con{2}>0$ and $\con{3}$ depending only on
$\con{1}$, $\epsilon$ and $\delta$, and $\con{4}$ depending only on the action,
and a homogeneous quasimorphism $\phi$ on $G$ 
satisfying the following properties:
\begin{enumerate}
\item{$\phi(g)\ge n\con{2}/\log{n}$;}
\item{$D(\phi)\le\con{3}$;}
\item{$|\phi(h)|\le 2|h|_S\con{4}/\log{n}$ for any $h \in G$.}
\end{enumerate}
\end{proposition}
\begin{proof}
Let $K$ be such that any $(2,4)$-quasigeodesic in a $\delta$-hyperbolic space 
stays within distance $K$ of a genuine geodesic. Choose $K$ with this
property, and let $\Sigma$ be the set of $\gamma_i$ that do not $K$ anti-align some
$\gamma' \subset \gamma$. Let $c_\Sigma$ denote the small counting function associated
to the set $\Sigma$, and similarly $c_{\Sigma^{-1}}$.
Then by Lemma~\ref{Fujiwara_antialign_lemma} and the definition of
$\Sigma$ we have $c_\Sigma(g) \ge n\con{2}/\log{n}$ and $c_{\Sigma^{-1}}(g)=0$.
Then let $\phi$ be the homogenization of $h_\Sigma$.

To obtain the last bullet point, observe that $d(p,h(p))\le \con{5}|h|_S$ for any
$h\in G$, for some constant $\con{5}$.
\end{proof}

Obtaining lower bounds on $\scl$ for random walks thus reduces to showing that for
certain $G$, the condition in Proposition~\ref{lower_bound_template} holds with
high probability for $g$ the result of a random walk of length $n/L$ (where $L$ is the
drift). The main technical issue is to relate subwalks of a random walk with subpaths of
the geodesic joining the endpoints. We address this point in what follows.

\subsection{Nondegenerate random walks}

We fix a group $G$ and probability measure $\mu$ supported on a finite subset $S$
of $G$ so that $S=S^{-1}$, and $\mu(s)=\mu(s^{-1})$ for each $s\in S$ (such a measure is said to
be {\em symmetric}). We also fix a
simplicial action of $G$ on a $\delta$-hyperbolic simplicial complex $Y$ (not assumed
to be locally finite) with a basepoint $p$.

The subset $S$ generates some subgroup $\langle S\rangle$ of $G$; we say $\mu$ is 
{\em nonelementary} if $\langle S \rangle$ does not fix any finite subset of $\partial_\infty Y$.

A random sequence $\id=g_0, g_1, g_2,\cdots$ (finite or infinite) is a
{\em $\mu$-random walk in $G$} if the successive differences
$s_i:=g_{i-1}^{-1}g_i$ are independent random elements of $S$ each
with the distribution $\mu$. For such a walk, define the associated
{\em $\mu$-random walk in $Y$}, namely the sequence $p=p_0, p_1,p_2
\cdots$ where $p_i = g_ip$. Since $S$ is finite, the set of distances
$d_Y(p_i,p_{i+1})$ is uniformly bounded; we say such a walk has {\em
  bounded increments}.

For any measure $\mu$ of finite first moment on any group acting
isometrically on any metric space, Kingman's subadditive ergodic
theorem implies that there is a constant $L\ge 0$ called the {\em
  drift} so that $L=\lim_{n \to \infty} d(p,p_n)/n$ almost surely.
  
Recall that for any points $p,q$ in a hyperbolic space $Y$ 
and any constant $K$, the {\em shadow} $S_p(q,K)$ is defined to be the set of
all points $r \in Y$ so that every geodesic from $p$ to $r$ comes within distance $K$
of $q$. This maybe expressed equivalently (up to slightly adjusting the
constant $R$) in terms of the ``Gromov product'' 
$(x\cdot y)_p:=1/2(d_Y(p,x)+d_Y(p,y)-d_Y(x,y))$ as the set of points with
$(q\cdot r)_p \ge d_Y(p,q)-R$.

We will make the following assumptions about $G,Y,\mu$.

\begin{definition}\label{def:nondegenerate}
Let $G$ be a group acting simplicially on a $\delta$-hyperbolic simplicial complex $Y$,
and let $\mu$ be a symmetric probability measure with support equal to
some set $S \subset G$.
A triple $(G,\mu,Y)$ as above is {\em nondegenerate} if it satisfies the following conditions:
\begin{enumerate}
\item{{\bf (nonelementary:)} $\langle S \rangle$ does not fix any finite subset of $\partial_\infty Y$;}
\item{{\bf (positive drift:)} the drift $L$ is positive;}
\item{{\bf (acylindricity:)} for any $G$-orbit $Gp\subset Y$ and for any $\con{1}>0$ there are constants $\con{2},\con{3}$ so that if
$q,r$ are points in $Gp$ with $d_Y(q,r)\ge \con{2}$, there are at most $\con{3}$ elements $g \in G$ 
with $d_Y(q,gq)\le \con{1}$ and $d_Y(r,gr)\le \con{1}$;}
\item{{\bf (linear progress:)} there are constants $\con{1}>1$, $\con{2}>0$ so that 
$$\Pr(d_Y(p_0,p_n) \in [L\con{1}^{-1}n,L\con{1}n]) \ge 1 - e^{-n\con{2}}$$}
\item{{\bf (exponential decay:)} there are constants $\con{1}$ and
$\con{2}>0$ such that for any $y \in Y$ and any $K$, the probability
that the result of a random walk of length $n$ lies in the shadow $S_{p_0}(y, K)$
decays exponentially in the distance to the shadow; i.e.
$$\Pr(p_n \in S_{p_0}(y,K)) \le \con{1}e^{-\con{2}(d_Y(p_0,y)-K)}$$}
\end{enumerate}
\end{definition}

\begin{remark}
We do not claim that every condition in this list is logically necessary; rather it
reflects the ingredients that go into our proof of Theorem~\ref{thm:hyperbolic_lower_bound}. 
It is natural to wonder whether the condition of acylindricity could be replaced by 
Bestvina--Fujiwara's {\em weakly properly discontinuous} condition
(see \cite{Bestvina_Fujiwara}, p.~76), since the latter condition is known to hold for
a wider class of group actions; but our arguments do not seem to easily allow it.
\end{remark}

\begin{remark}\label{orbit_acylindricity_remark}
The condition that we call ``acylindricity'' in Definition~\ref{def:nondegenerate} is weaker
than what is usually called acylindricity for an action, in that the constants $C_2$, $C_3$
are allowed to depend on both $C_1$ and the choice of $G$-orbit $G_p$. If a distinction
needs to be made, we refer to our weaker notion as {\em orbit acylindricity}.
If the action of $G$ on $Y$ is cocompact, both notions of acylindricity are equivalent. 
\end{remark}

The following proposition is largely obtained
by assembling known results:

\begin{proposition}\label{prop:nondegenerate_examples}
In each of the following cases, $(G,\mu,Y)$ is nondegenerate:
\begin{enumerate}
\item{$G$ is a hyperbolic group, $Y$ is the Cayley graph of $G$ with respect to some finite
generating set, and $\langle S \rangle$ is nonelementary.}
\item {$G$ is a (strongly) relatively hyperbolic group, $Y$ is the
    Groves--Manning space associated to $G$, and $\langle S \rangle$ is
  nonelementary.}
\item{For some surface $\Sigma$, the group $G$ is the mapping class group, $Y$ is 
the complex of curves, and $\langle S \rangle$ is not reducible or virtually abelian.}
\end{enumerate}
\end{proposition}

\begin{proof}
The nonelementary axiom follows by hypothesis in all three cases. 
Exponential decay in case (3) follows from Maher \cite{Maher_exp}, and
again the arguments go through verbatim in the setting of a
non-elementary action of a group on a proper Gromov hyperbolic space,
so cases (1) and (2) also follow from this, though presumably the
result is standard in case (1). We now verify the other properties.

In case (1), the Cayley graph $Y$ is a proper Gromov hyperbolic space,
and acylindricity follows from the fact that the action of $G$ on
itself is properly discontinuous. Positive drift holds in cases (1) by
Kaimanovich \cite{Kaimanovich}, Thm.~7.3, and linear progress follows
from Kesten's estimate \cite{Woess}, Lem.~8.1b for a random walk on a
nonamenable group.

In case (2), given a (strongly) relatively hyperbolic group $G$,
Groves and Manning \cite{Groves_Manning} construct a \emph{proper}
Gromov hyperbolic space $Y$ (called the
{\em cusped space} --- see \cite{Groves_Manning} \S~3)
on which $G$ acts by isometries properly
discontinuously, but not cocompactly; and orbit acylindricity follows from
the fact that the space $Y$ is proper, and the action of $G$ on $Y$ is
properly discontinuous. Positive drift holds by Kaimanovich
\cite{Kaimanovich}, Thm.~7.3, and linear progress follows from Maher
\cite{Maher_exp}.  Although the results of \cite{Maher_exp} are stated
in terms of the action of the mapping class group on the complex of
curves, the results hold in the (simpler) case of an action of a
non-elementary group on a proper Gromov hyperbolic space $Y$.

In case (3), the complex of curves is a locally infinite Gromov
hyperbolic simplicial complex on which the mapping glass group acts
discontinuously by simplicial isometries, and acylindricity is a
theorem of Bowditch \cite{Bowditch}, Thm.~1.3. Positive drift and
linear progress follow from Maher \cite{Maher_linear}, Thm.~1.1.
\end{proof}

\subsection{Proximal points and unfolded walks}

Since $S$ is finite, the length of successive steps $d(p_i,p_{i+1})$ is uniformly
bounded by a constant, and therefore we can think of the random walk as a (coarse)
path in $Y$ of length $\le \con{1}n$. We would like to use this estimate to show
that with very high probability, ``most'' of the $p_i$ are within a bounded distance
of the geodesic from $p_0$ to $p_n$. Actually, it turns out to be easier (and just
as useful) to show that most points on the geodesic from $p_0$ to $p_n$ are within
a bounded distance of some $p_i$, and moreover this fact can be deduced directly from
linear progress (see Definition~\ref{def:nondegenerate}) and elementary hyperbolic geometry.

\begin{definition}\label{def:bounded_increments}
Let $p_0,p_1,\cdots,p_n$ be a walk on a $\delta$-hyperbolic space $Y$. If for all $i$ there is
an inequality $d(p_i,p_{i+1})\le \con{1}$ we say $p_i$ has {\em $\con{1}$-bounded increments}.
\end{definition}

\begin{definition}\label{def:proximal_point}
Let $p_0,p_1,\cdots,p_n$ be a walk on a $\delta$-hyperbolic space $Y$, and
let $\gamma$ be a geodesic from $p_0$ to $p_n$. A point $q\in \gamma$ is {\em $K$-proximal}
if $d_Y(q,p_i) \le K$ for some $p_i$. We denote the $K$-proximal subset of $\gamma$ by $\gamma_K$.
\end{definition}

\begin{lemma}\label{lem:most_points_proximal}
Let $p_0,p_1,\cdots,p_n$ be a walk on a $\delta$-hyperbolic space $Y$ with $\con{1}$-bounded
increments. Let $\gamma$ be a geodesic from $p_0$ to $p_n$. Suppose $\length(\gamma)\ge \con{2}n$.
Then for any $\epsilon > 0$ there is a constant $K(\epsilon,\con{1},\con{2})$ so that
if $\gamma_K$ denotes the $K$-proximal subset of $\gamma$, there is an estimate
$$\length(\gamma_K)/\length(\gamma) \ge 1-\epsilon$$
\end{lemma}
\begin{proof}
We assume $K\gg C \gg \delta$ for convenience.

There is a nearest point projection $\pi$ from the $p_i$ to $\gamma$
so that $d(\pi(p_i),\pi(p_{i+1})) \le C+O(\delta)$ and consequently
every point in $\gamma$ is within distance $C/2 + O(\delta)$ from some
$\pi(p_i)$.  If $p \in \gamma$ is not $K$-proximal, then
$d(p_i,\gamma) \ge K-O(C)$ and therefore $d(p_j,\gamma) \ge 3K/4-O(C)$
for $|j-i|\le K/4C$. On the other hand, $d(p_i,p_j) \le |j-i|C$ so any
geodesic from $p_j$ to $p_i$ does not come within distance $K/2$ of
$\gamma$. By $\delta$-thinness, we can conclude that
$d(\pi(p_j),\pi(p_i)) = O(\delta)$ for $|j-i| \le K/4C$.

If the set of non $K$-proximal points has length at least $\epsilon\cdot \length(\gamma)$, 
there are at least $\epsilon\cdot \length(\gamma)/O(\delta)$ such points whose mutual pairwise
distances is at least $O(\delta)$. To each such point we can associate a sequence of
$K/2C$ points $p_i$ whose projections to $\gamma$ are within $O(\delta)$ of it, and therefore
these collections of points are {\em disjoint}. The total number of $p_i$ in these collections
is at least $(K/2C) \epsilon C^{-1}n/O(\delta) = nK\epsilon C^{-2}/O(\delta)$ so if 
$K> \epsilon^{-1} C^2 O(\delta)$ we get a contradiction, as desired.
\end{proof}

From the linear progress axiom, we deduce the following:

\begin{lemma}[proximal]\label{lem:close_on_big_subset}
Let $(G,\mu,Y)$ be nondegenerate. Let $p_0,\cdots,p_n$ be a random walk, and $\gamma$ a geodesic
from $p_0$ to $p_n$. There are constants $\con{1},\con{2}>0$ so that for any $\epsilon > 0$ there is a
further constant $K(\con{1},\epsilon)$ so that 
$$\Pr\left((\length(\gamma)\ge \con{1}n) \wedge (\length(\gamma_K)/\length(\gamma) \ge 1-\epsilon)\right) \ge 1 - e^{-n\con{2}}$$
\end{lemma}
Note that the first condition implies the second by Lemma~\ref{lem:most_points_proximal}, but it
is convenient to state both conditions explicitly.

\begin{definition}\label{def:long_geodesic}
We say $\gamma$ as above is {\em $\con{1}$-long} (or just {\em long} if $\con{1}$ is understood)
if $\length(\gamma)\ge \con{1}n$.
\end{definition}

In the sequel we use the convention that oriented geodesics are parameterized proportional to
arclength. We also use the convention that $\gamma$ is oriented from $p_0$ to $p_n$, so we write
$\gamma(0)=p_0$ and $\gamma(1)=p_n$. 

Now, by definition, for every $s,t$ with $\gamma(s),\gamma(t)\in \gamma_K$
there are indices $i,j$ so that 
$d_Y(p_i,\gamma(s))\le K$ and $d_Y(p_j,\gamma(t))\le K$. It is not necessarily true, however, that $s<t$ implies
$i<j$. Nevertheless, this should be true whenever $d_Y(\gamma(s),\gamma(t))$ is sufficiently large,
with big probability. We quantify this.

\begin{definition}\label{def:unfolded}
Fix some big constant $M\gg 0$. We say a walk $p_0,\cdots, p_n$ is {\em $(K,M\log{n})$-unfolded} (or just
{\em unfolded} if $K$ and $M$ are understood) if for every $j>i$ with $j-i>M\log{n}$ and every geodesic
$\gamma_i$ from $p_0$ to $p_i$, we have $d_Y(p_j,\gamma_i)>K$.
\end{definition}

\begin{lemma}[unfolded]\label{lem:unfolded}
Let $(G,\mu,Y)$ be nondegenerate. Let $p_0,\cdots, p_n$ be a random walk. For any $K$ there is
$\con{1}>0$ so that
$$\Pr(\text{walk is $(K,M\log{n})$-unfolded})\ge 1-n^{2-\con{1}M}$$
\end{lemma}
\begin{proof}
Pick indices $i,j$ so that $j>i$ and $j-i>M\log{n}$. By the linear
progress property, the walk $p_i,p_{i+1},\cdots,p_j$ satisfies
$d_Y(p_i,p_j)>\con{1}M\log{n}$ with probability $1-n^{-\con{2}M}$ for
suitable $\con{1},\con{2}$,. The subset of the $K$-neighborhood of
$\gamma_i$ which lies outside of $B(p_i, \con{1}M \log{n})$ is contained
in a shadow $S_{p_i}(y , K + \con{3})$, where $y$ is the point
distance $\con{1} M \log{n}$ from $p_i$ along $\gamma_i$, and
$\con{3}$ only depends on $\delta$, the constant of hyperbolicity. 
By linear progress and the the exponential decay property for shadows, the probability that $p_j$ is
within distance $K$ of $\gamma_i$ is at most $n^{-\con{4}M}$, for some
constant $\con{4}$. Hence
$$\Pr\left((d_Y(p_i,p_j)>\con{1}M\log{n})\wedge(d_Y(p_j,\gamma_i')>K)\right)>1-n^{-\con{5}M}$$
There are fewer than $n^2$ indices $i,j$ as above; the lemma follows.
\end{proof}

\begin{remark}
The condition that a walk should be $(K,M\log{n})$-unfolded is very strong, probably much stronger
than we need. But it does simplify the proofs to come. If we just insist that {\em most} pairs of
indices $i,j$ with $j-i>M\log{n}$ are ``unfolded'' in the obvious sense, then the probability
will be at least $1-e^{-n\con{1}}$, for sufficiently big (but fixed) $M$. We do not use this fact
in the sequel.
\end{remark}

\subsection{Anti-aligned segments and matching}

Let $p_0,p_1,\cdots,p_n$ as above be a walk in $Y$ with bounded increments, and $\gamma$ the 
geodesic from $0$ to $n$. We assume $\gamma$ is long, and let $N=\length(\gamma)$. We
fix some constants $M\gg 0$ and $K\gg 0$ (to be determined later).

\begin{definition}
An {\em $R$-match} (or just a {\em match} if $R$ is understood)
is a triple $(\alpha,\beta,h)$ where $\alpha,\beta$ are geodesic segments
of $\gamma$ of length $R$, and $h \in G$ so that $h\alpha$ {\em anti-aligns} $\beta$; 
i.e.\/ $d(h\alpha(0),\beta(1)) \le K$ and $d(h\alpha(1),\beta(0))\le K$.
\end{definition}

Ultimately we will be concerned with $R$-matches where $R=M\log{n}$. It is convenient for our
$R$-matches not to be too close to each other; that can be achieved by the following ``cut-in-half'' lemma:

\begin{lemma}[cut in half]\label{lem:trim_match}
Let $(\alpha,\beta,h)$ be an $R$-match, and let $\gamma_K \subset \gamma$ so that
$\length(\gamma_K \cap \alpha) \ge (1-\epsilon)R$ and 
$\length(\gamma_K \cap \beta) \ge (1-\epsilon)R$.
Then there are subsegments $\alpha' \subset \alpha$
and $\beta' \subset \beta$ with endpoints in $\gamma_K\cap \alpha$ and $\gamma_K\cap\beta$
respectively, so that $(\alpha',\beta',h)$ is an $R'$-match and 
$d_Y(\alpha',\beta')\ge R/2$, for some $R'/R \in [1/4-4\epsilon,1/4+4\epsilon]$.
Furthermore $\length(\gamma_K \cap \alpha') \ge (1-5\epsilon)R'$ and similarly for $\beta'$.
\end{lemma}
\begin{proof}
If $\nu$ is an oriented geodesic, let $\nu^-$ denote the initial half of $\nu$, and $\nu^+$
the terminal half. One of the pairs $(\alpha^+,\beta^-)$ and $(\alpha^-,\beta^+)$ must be disjoint,
and therefore one of the pairs $(\alpha^{++},\beta^{--}),\cdots,(\alpha^{--},\beta^{++})$ must 
have segments separated from each other by distance
at least $R/2$. Each of these segments has length $R/4$, so the length of their intersections with
$\gamma_K$ are both at least $R/4-\epsilon R$, and there must be further subsets in each
of length at least $R/4-2\epsilon R$ matched by $h$. Let $\alpha',\beta'$ be maximal subsegments with
endpoints in these subsets.
\end{proof}

\begin{definition}\label{def:well_matched}
Let $\gamma$ be a geodesic in $Y$ of length $N$.
Let $\alpha_i$ be successive subpaths of $\gamma$ of length $M\log{N}$, so that there are
$N/M\log{N}$ of them. We say $\gamma$ is {\em well-matched} if for {\em every} subset $I$ of
indices with $|I|>N/10 M\log{N}$ there are at least $9|I|/10$ disjoint geodesics $\beta_j$
in $\gamma$ and elements $h_j \in G$ and indices $i(j) \in I$ so that $(\alpha_{i(j)},\beta_j,h_j)$
is a match.
\end{definition}

Note that the definition of a match and well-matched implicitly depend on $K$ and $M$.
If we need to specify them we use the terminology {\em $(K,M\log{N})$ well-matched}.

\begin{lemma}\label{lem:no_matching_quasimorphism}
Let $G$ act simplicially on a $\delta$-hyperbolic simplicial space $Y$.
Let $p$ be a basepoint, and let $\gamma$ be a geodesic from $p$ to $g(p)$ for some
$g\in G$. Then there are constants $K$ and $\con{1}$
depending only on $\delta$, so that if $\gamma$ is not $(K,M\log{N})$ well-matched, there is a
counting quasimorphism $\phi$ on $G$, supported on words of length
$M\log{N}$, with $D(\phi) \le \con{1}$ and $\phi(g)\ge N/(100M\log{N})$.
\end{lemma}
\begin{proof}
Choose a set of indices $I$ as above witnessing $\gamma$'s failure to be $K$ well-matched.
Then no realizing path for $g$ can contain more than $9|I|/10$ disjoint anti-aligned copies of $K$.
Define $c_I$ to be the counting function which counts {\em disjoint} copies of any of the $\alpha_i$ with
$i \in I$ in a realizing path. We can take $\gamma$ itself as a realizing path for $c_I$,
so that $c_I(g) \ge |I|$. On the other hand, any realizing path for $g^{-1}$ for $c_I$ contains at most
$9|I|/10$ disjoint copies of $\alpha_i$, so $\phi(g):=c_I(g) - c_I(g^{-1}) \ge |I|/10$.
\end{proof}

\begin{remark}
We have implicitly used the fact that disjoint copies of $\alpha_i^{-1}$ on
any realizing path for $g^{-1}$ are close to geodesic segments of $\gamma$ which overlap only
in segments of length $\le 2K+O(\delta)$; trimming these overlaps, we can assume the nearby
segments in $\gamma$ are disjoint.
\end{remark}

It remains to understand when $\gamma$ associated to a random walk is $(K,M\log{N})$ well-matched, 
for suitable constants $K,M$. 

The next lemma shows that any well-matching can be ``relativized'' to any subset $\gamma_K$ with
$\length(\gamma_K)/\length(\gamma)\ge 1-\epsilon$.

\begin{lemma}\label{lem:pigeonhole}
Let $\gamma$ be a geodesic of length $N$, and let $\beta_j$ a collection of at least $9N/(100M\log{N})$
disjoint subpaths each of length $M\log{N}$. 
Let $\gamma_K$ be a subset of $\gamma$ with $\length(\gamma_K)/\length(\gamma) \ge 1-\epsilon$.
Then at least $(1-\sqrt{\epsilon})$ of the $\beta_j$ satisfy 
$$\length(\gamma_K \cap \beta_j)/\length(\beta_j) \ge 1-12\sqrt{\epsilon}$$
\end{lemma}
\begin{proof}
This is essentially just Chebyshev's inequality. 
The total length of the $\beta_j$ is at least $9N/100$.
For each $\beta_j$ which fails to satisfy the desired inequality, 
$\length(\gamma_K^C \cap \beta_j) \ge 12\sqrt{\epsilon} \cdot\length(\beta_j)$
and therefore 
$$\epsilon N \ge \length(\gamma_K^C) \ge 12\sqrt{\epsilon}\delta 9N/100$$ 
where $\delta$ is the proportion of ``failing'' $\beta_j$. We conclude $\delta \le \sqrt{\epsilon}$
as claimed.
\end{proof}

Combining Lemma~\ref{lem:pigeonhole} with Lemma~\ref{lem:trim_match} we deduce the following:

\begin{proposition}[relative well-matching]\label{prop:relative_well_matching}
Let $\gamma$ be a geodesic in $Y$ of length $N$.
Suppose $\gamma$ is $(K,M\log{N})$ well-matched, and let $\gamma_K$ be a subset of $\gamma$
with $\length(\gamma_K)/\length(\gamma)\ge 1-\epsilon$. 

If $\epsilon$ is sufficiently small,
there are at least $8N/(10M\log{N})$ segments $\alpha_j$ in $\gamma$, and $8N/(10M\log{N})$
{\em disjoint} segments $\beta_j$ in $\gamma$, each of length at least $M\log{N}/5$, so that for each
$\beta_j$ there is a match $(\alpha_j,\beta_j,h_j)$ satisfying
\begin{enumerate}
\item{the endpoints of $\beta_j$ are in $\gamma_K$ and similarly for $\alpha_{i(j)}$;}
\item{there is an inequality $\length(\gamma_K \cap \beta_j)/\length(\beta_j)\ge 1-100\sqrt{\epsilon}$
and similarly for $\alpha_{i(j)}$; and}
\item{$d_Y(\beta_j,\alpha_{i(j)})\ge M\log{N}/2$.}
\end{enumerate}
\end{proposition}
\begin{proof}
Of all the successive segments $\alpha_i$ of $\gamma$ of length $M\log{N}$, let $I'$ be
the set of indices for which $\length(\gamma_K\cap \alpha_i)/\length(\alpha_i)\ge 1-12\sqrt{\epsilon}$.
Then $|I'|\ge (1-\sqrt{\epsilon})N/(M\log{N})$ by Lemma~\ref{lem:pigeonhole}.
Since $\gamma$ is well-matched, we can find at least $9|I'|/10$ disjoint geodesics $\beta_j$ in $\gamma$
and matches $(\alpha_{i(j)},\beta_j,h_j)$ with $i(j)\in I'$. Applying Lemma~\ref{lem:pigeonhole}
again, there is a subset $I \subset I'$ with $|I|\ge (1-\sqrt{\epsilon})^2N/(M\log{N})$ so that
further $\length(\gamma_K\cap \beta_j)/\length(\beta_j)\ge 1-12\sqrt{\epsilon}$.

For each $\beta_j$, choose some $\alpha_{i(j)}$ as above, and relabel it as $\alpha_j$. Note that
although the $\beta_j$ are distinct for different $j$, we do {\em not} assume the $\alpha_j$ are
distinct for different $j$ (for all we know, they might all be the same!).
Applying Lemma~\ref{lem:trim_match} to each $(\alpha_j,\beta_j,h_j)$ gives rise to new
$(\alpha_j',\beta_j',h_j)$ of at least $1/4-100\sqrt{\epsilon}$ the length, satisfying the desired
properties. Notice now that even if $\alpha_i = \alpha_j$ we might have $\alpha_i' \ne \alpha_j'$.
\end{proof}

\subsection{Uniform lower bounds for hyperbolic spaces}

We are now ready to return to probability. 

\begin{lemma}\label{lem:random_matching}
Let $(G,\mu,Y)$ be nondegenerate, and let $p_0,\cdots,p_n$ be a random walk of length $n$.
Fix $\con{1}$ and $K$. Then there is a constant $\con{2}$ so that for any $M$ the following holds.
Consider the collection of indices $a<a'<b<c<c'<d$ for which there are 
geodesics $\alpha$ from $p_a$ to $p_b$ and
$\beta$ from $p_c$ to $p_d$ with the following properties:
\begin{enumerate}
\item{$\length(\alpha)\ge M\con{1}\log{n}$ and similarly for $\beta$;}
\item{there is $t \in [0.1,0.2]$ so that $d_Y(p_{a'},\alpha(1-t))\le K$ 
and $d_Y(p_{c'},\beta(t))\le K$;}
\item{there is some $h\in G$ so that $d_Y(h\alpha(0),\beta(1))\le K$ 
and $d_Y(h\alpha(1),\beta(0))\le K$.}
\end{enumerate}
The probability that this collection of indices is nonempty is at most $n^{6-\con{2}M}$.
\end{lemma}
\begin{proof}
Fix a collection of indices $a<a'<b<c<c'<d$. The key point is that we don't care about
$\alpha$ and $\beta$ {\it per se}, but only on their $G$-orbits $G\alpha$ and $G\beta$, since
$\beta$ is only compared with a translate $h\alpha$. Speaking somewhat loosely, 
since $b<c$, the random variables
$G\alpha$ and $G\beta$ are {\em independent}. In fact, if we let $\beta''$ be a geodesic 
from $p_c$ to $p_{c'}$, and $\beta'$ a geodesic from $p_{c'}$ to $p_d$
then the orbit $G\beta'$ is independent of $G\alpha$ and $G\beta''$.

By orbit acylindricity (this is the only place in the entire argument that
acylindricity is used), there are only {\em boundedly} many $h \in G$
with $d(h\alpha(1),\beta(0))\le K$ and $d(h\alpha(1-t),\beta(t))\le
3K$, providing $n$ is bigger than some universal constant. So we can
suppose that we are given a {\em fixed} geodesic $\alpha'$ of length
$\ge (1-t)M\con{1}\log{n}$ starting within distance $3K$ of $p_{c'}$,
and we want to estimate the probability that a random walk from
$p_{c'}$ to $p_d$ ends up within distance $K$ of the endpoint
$h\alpha(0)$. This can be bounded from above by the probability that a
random walk started at $p_{c'}$ lies in the finite collection of
shadows $S_{p_{c'}}(h \alpha(0), K + \con{2})$, for some constant
$\con{2}$ which only depends on $\delta$, the constant of
hyperbolicity.  By the exponential decay property for shadows, this
probability is at most $n^{-\con{3}M}$ for some $\con{3}$.

Summing over all possible choices of indices gives probability at most
$n^{6-\con{3}M}$.
\end{proof}

We deduce our main theorem on nondegenerate random walks on hyperbolic spaces.

\begin{theorem}[Hyperbolic lower bound]\label{thm:hyperbolic_lower_bound}
Let $Y$ be a $\delta$-hyperbolic simplicial complex (not assumed to be locally finite),
and let $G$ be a finitely generated group that
acts simplicially on $Y$. Let $\mu$ be a symmetric probability measure of
finite support on $G$ so that $(G,\mu,Y)$ is nondegenerate, in the sense of 
Definition~\ref{def:nondegenerate}.

Let $g$ be obtained by random walk on $G$ (with respect to $\mu$) of length $n$. Then
for any $\con{1}$ there is a $\con{2}>0$ and $\con{3}$ depending only on $\delta$ and $G$,
so that with probability at least $1-n^{-\con{1}}$ there is a homogeneous quasimorphism
$\phi$ on $G$ satisfying the following properties:
\begin{enumerate}
\item{$\phi(g)\ge n\con{2}/\log{n}$;}
\item{$D(\phi)\le \con{3}$;}
\item{$|\phi(h)|\le 2d_Y(q,hq)\con{3}/\log{n}$ for any $g \in G$ and any $q\in Y$.}
\end{enumerate}
In particular, for any $\con{1}>0$ there is a constant $C>1$ 
so that if we condition on $g \in [G,G]$ (for $n$ even), then
$$\Pr(C^{-1}n/\log{n} \le \scl(g) \le Cn/\log{n}) \ge 1-n^{-\con{1}}$$
\end{theorem}
\begin{proof}
We use the notation $K$, $M$ as above.

Let $p_0,\cdots, p_n$ be the associated walk from the basepoint $p$ to $gp$, and let
$\gamma$ be the geodesic from $p$ to $gp$. We assume that
$\gamma$ is long (see Definition~\ref{def:long_geodesic}), that the walk is unfolded 
(see Definition~\ref{def:unfolded}), and that the conclusions of 
Lemma~\ref{lem:random_matching} hold. By Lemma~\ref{lem:close_on_big_subset},
 Lemma~\ref{lem:unfolded} and Lemma~\ref{lem:random_matching},
the probability of all three of these things happening
is at least $1-n^{6-\con{1}M}$ for some $\con{1}$, where $M$ is as above.

Since $\gamma$ is long, $\length(\gamma)=Cn$ for some constant $C$ bounded above and below.
If $\gamma$ is not $(K,M\log{(Cn)})$ well-matched, the desired conclusion follows from 
Lemma~\ref{lem:no_matching_quasimorphism}, so we suppose $\gamma$ is well-matched.

By Proposition~\ref{prop:relative_well_matching} we can find a pair of segments (in fact, many
such pairs) $\alpha$, $\beta$ of length $M\log{(Cn)}/5$, with endpoints in $\gamma_K$ with 
$\length(\gamma_K\cap \alpha)/\length(\alpha) \ge 1-100\sqrt{\epsilon}$ (and similarly for
$\beta$), and which are distance at least $\ge M\log{(Cn)}/2$ apart in $\gamma$. Pick
$t\in [0.1,0.2]$ with $\alpha(1-t)$ and $\beta(t)$ in $\gamma_K$. Since $\gamma$ is
unfolded, the points $\alpha(0)$, $\alpha(1-t)$, $\alpha(1)$, $\beta(0)$, $\beta(t)$, 
$\beta(1)$ are within distance $K$ of points $p_a$, $p_{a'}$, $p_b$, $p_c$, $p_{c'}$, $p_d$ 
on the walk so that the indices appear in the same order that the corresponding 
points appear in $\gamma$.

But the conclusion of Lemma~\ref{lem:random_matching} says that no such pairs 
$\alpha$, $\beta$ can exist, so we get a contradiction. It follows that $\gamma$ is
not well-matched after all, and the desired quasimorphism exists.
\end{proof}

From Proposition~\ref{prop:nondegenerate_examples} we obtain the following corollaries:

\begin{corollary}
Let $G$ be either a hyperbolic group (resp. the mapping class group of a surface $\Sigma$), 
and let $\mu$ a symmetric probability measure on $G$
with finite support generating a nonelementary subgroup (resp. a subgroup which 
is not reducible or virtually abelian). Then for any $\con{1}$ there is $C$
so that if $g$ is obtained by random walk in $G$ of length $n$ (even)
and conditioned to lie in $[G,G]$ we have
$$\Pr(C^{-1}n/\log{n} \le \scl(g) \le Cn/\log{n}) \ge 1-n^{-\con{1}}$$
\end{corollary}

Since reducible elements in mapping class groups fix points in the complex of curves,
we obtain:

\begin{corollary}\label{cor:reducible_lower_bound}
Let $\mu$ be a symmetric probability measure on the mapping class group of $\Sigma$
of finite support, and suppose the subgroup it generates is not reducible or virtually
abelian. Then for any $\con{1}$ there is $C$ so that if $g$ is obtained by random walk
in $G$ of length $n$, then $g$ cannot be expressed as a product of fewer than
$Cn/\log{n}$ reducible elements, with probability at least $1-n^{-\con{1}}$.
\end{corollary}

From monotonicity of $\scl$ under homomorphisms, we obtain:

\begin{corollary}\label{cor:growth_obstruction}
Let $G$ be a group and $\mu$ a symmetric probability measure with finite support which
generates $G$. Suppose that for any $\epsilon>0$ there is a $\delta>0$ so that if $g$ is
obtained by random walk on $G$ of length $n$ (even) conditioned to lie in $[G,G]$, we have
$$\Pr(\scl(g)<\epsilon\cdot n/\log{n})>\delta$$
Then every homomorphism from $G$ to a hyperbolic group or to a mapping class group has
virtually abelian image.
\end{corollary}
\begin{proof}
This just depends on the observation that reducible subgroups of mapping class groups are
themselves mapping class groups of simpler surfaces.
\end{proof}

\subsection{Geometry of the scl norm of a random subspace}\label{random_norm_subsection}

As remarked earlier, stable commutator length is not merely a function, but actually
defines a pseudo-norm on the space $B_1(G)$ of (real) group $1$-boundaries; i.e.\/
formal real linear combinations 
$\sum t_i g_i$ with $t_i \in \R$ and $g_i \in G$, representing $0$ in $H_1(G;\R)$.
In fact, $\scl$ descends to a pseudo-norm on the ``homogenized'' quotient 
$B_1^H(G):=B_1(G)/\langle g-hgh^{-1}, g^n-ng\rangle$ --- see 
\cite{Calegari_scl} \S~2.6 for details.

A nice corollary of Theorem~\ref{thm:hyperbolic_lower_bound} is to obtain {\it a priori}
geometric control over the geometry of a random subspace of $B_1^H(G)$ of fixed dimension.

\begin{theorem}[Random norm ball theorem]\label{thm:random_norm_ball}
Let $Y$ be a $\delta$-hyperbolic simplicial complex (not assumed to be locally finite),
and let $G$ be a finitely generated group that
acts simplicially on $Y$. Let $\mu$ be a symmetric probability measure of
finite support on $G$ so that $(G,\mu,Y)$ is nondegenerate, in the sense of 
Definition~\ref{def:nondegenerate}.

Fix $k$ and let $g_1,g_2,\cdots,g_k$ 
be obtained by random walk on $G$ (with respect to $\mu$) of (even) lengths 
$\ell_1n, \ell_2n,\cdots, \ell_kn$, all conditioned to lie in $[G,G]$. Then
for any $\con{1}$ there is a $\con{2}>0$ and $\con{3}$ depending only on $\delta$ and $G$,
so that with probability at least $1-n^{-\con{1}}$, for any formal sum
$\sum t_i g_i$ there is an estimate
$$\Pr(C^{-1}(\sum_i t_i\ell_i n)/\log{n} \le \scl(\sum t_i g_i) \le C(\sum_i t_i\ell_i n)/\log{n}) \ge 1-n^{-\con{1}}$$
\end{theorem}
\begin{proof}
We give the argument in the case $k=2$ and $\ell_1=\ell_2=1$. 
The general case follows by a minor modification of the argument.

We know that $\scl(g_1)$ and $\scl(g_2)$ are between $C^{-1}n/\log{n}$ and $Cn/\log{n}$ with
high probability. By the definition of a norm, $\scl(g_1 + g_2) \le C2n/\log{n}$. 
By the argument of Theorem~\ref{thm:hyperbolic_lower_bound}, again with high probability
we can find subsets $I_i$ ($i=1,2$)
of the indices associated to the quasimorphisms $\phi_i$ certifying the lower bound for
$\scl(g_i)$ so that $|I_i|/|I|\ge (1-\epsilon)$ and such that if $\phi_{1,2}$ is the
small counting quasimorphism associated to the union of the $I_i$ segments, then $\phi_{1,2}(g_i)\ge
(1-\epsilon)\phi_i(g_i)$. Note that $D(\phi_{1,2})\le \con{3}$ for the same constant as in 
Theorem~\ref{thm:hyperbolic_lower_bound}, since this constant is universal for {\em any} small
counting quasimorphism. Moreover, $\phi_{1,2}(t_1 g_1 + t_2 g_2) = 
t_1\phi_{1,2}(g_1) + t_2\phi_{1,2}(g_2) \ge (t_1+t_2)C^{-1}n/\log{n}$. This gives the desired
estimate in the orthant where the $t_i$ are both positive. In an orthant in which $t_i$ is negative,
use $g_i^{-1}$ in place of $g_i$ together with the observation that $t_ig_i = -t_ig_i^{-1}$ in $B_1^H$.
\end{proof}

\subsection{Concentration versus compression}

It is natural, in view of Theorem~\ref{free_group_theorem} to believe that the estimates in
Theorem~\ref{thm:hyperbolic_lower_bound} and Theorem~\ref{thm:random_norm_ball} can be sharpened.
We make the following conjecture:

\begin{conjecture}[Concentration conjecture]\label{con:hyperbolic_concentration}
Let $(G,\mu,Y)$ be nondegenerate in the sense of 
Definition~\ref{def:nondegenerate}. Then there is a constant $C>0$ and $\con{1}>0$ so that
for any $\epsilon>0$, if $g$ is obtained by random walk on $G$ of length $n$ (even), conditioned to
lie in $[G,G]$, then
$$\Pr(|\scl(g)\log{n}/n - C|\le \epsilon) \ge 1 - e^{-n^{\con{1}}}$$
\end{conjecture}

Conjecture~\ref{con:hyperbolic_concentration} would imply the following conjecture about the 
geometry of the norm in a random subspace:

\begin{conjecture}[Norm conjecture]\label{con:norm_concentration}
Let $(G,\mu,Y)$ be nondegenerate in the sense of 
Definition~\ref{def:nondegenerate}. Then there is a constant $C>0$ and $\con{1}>0$ so that for
any finite integer $k$ and any $\epsilon>0$, if $g_1,\cdots,g_k$ are obtained by independent random walks on $G$
of length $\ell_1n,\cdots,\ell_kn$ (even), conditioned to lie in $[G,G]$, then
$$\Pr(|\scl(\sum t_i g_i)\log{n}/n - C\sum t_i\ell_i|\le \epsilon) \ge 1 - e^{-n^{\con{1}}}$$
\end{conjecture}

It is tempting to conjecture further that $C=\lambda/6$ where $\lambda$ is the entropy of $\mu$, but
this might be premature without first understanding a wider range of examples.

\subsection{Random walks on Out($F_n$)} \label{subsection:outfn}

In this section we summarize our current knowledge of the behavior of
random walks on Out($F_n$). In this context, it is natural to consider
random walks on the isometry group of a non-locally compact Gromov
hyperbolic simplicial complex. In this section we show that if $\mu$
is a probability distribution with finite support on the isometry
group of a non-locally compact Gromov hyperbolic simplicial complex
$Y$, whose support generates a discrete non-elementary subgroup, then
all of the axioms from Definition~\ref{def:nondegenerate} are
satisfied, except possibly for acylindricity. A key step is to show
that random walks converge to the boundary.

\begin{theorem} \label{theorem:convergence} %
Let $\mu$ be a probability measure with finite support on the isometry
group of a (not-necessarily proper) Gromov hyperbolic simplicial
complex $Y$ with basepoint $y_0$, whose support generates a discrete
non-elementary subgroup. Then almost every sample path $\{ g_n y_0 \}$
converges to the Gromov boundary $\partial Y$, and the hitting measure
$\nu$ is non-atomic, and is the unique $\mu$-stationary measure on
$\partial Y$.
\end{theorem}

The arguments from \cite{Maher_linear} and \cite{Maher_exp} then go
through in this case, to give linear progress, positive drift and
exponential decay. We can then show:

\begin{theorem}[Linear translation length]\label{theorem:linear_translation}
Let $G$ be a group of isometries of a simplicial Gromov hyperbolic
space $Y$, which is not necessarily locally compact. Let $\mu$ be a
probability distribution with finite support on $G$, such that the
support of $\mu$ generates a non-elementary subgroup of $G$. Then
there are constants $L > 0$ and $c < 1$ such that
\[ \Pr( \tau(w_n) \le Ln ) \le O(c^n), \]
where $w_n$ is the group element obtained by a random walk of length
$n$, and $\tau(w_n)$ is the translation length of $w_n$ acting on $Y$.
\end{theorem}

The group Out($F_n$) acts on a (locally infinite) simplicial complex,
the complex of free factors, which Bestvina and Feighn
\cite{Bestvina_Feighn} have shown is Gromov hyperbolic, so the results
above apply in this case. If the action of Out($F_n$) on the complex of
free factors is acylindrical, then this would imply that the growth
rate of scl for random walks in Out($F_n$) is $n / \log n$. However,
the acylindricity of the action of Out($F_n$) is still unknown.  If
an element of Out($F_n$) acts on the complex of free factors with
positive translation length, then it is fully irreducible with fully irreducible
powers (some use the terminology {\em iwip}) and so this
generalizes results of Rivin \cite{Rivin2} and Kowalski
\cite{Kowalski}.  Finally, we remark that Out($F_n$) also acts on the
complex of free splittings, which Handel and Mosher
\cite{Handel_Mosher} have shown is Gromov hyperbolic, but in this case
the action is definitely not acylindrical.

\begin{remark}
In a recent preprint, Sisto \cite{Sisto} introduces the notion of a
{\em weakly contracting} element in a group (with respect to a so-called
{\em path system}), and shows (for example)
that an element of Out($F_n$) obtained by (suitable) random walk will be weakly contracting
(and therefore iwip)
with probability going to 1 exponentially fast with the length of the walk.
Sisto proves many other interesting results in his preprint, with methods that
do not seem to overlap much with ours.
\end{remark}

\begin{remark}
Since this paper was first posted, more facts have been established about
the action of Out($F_n$) on various hyperbolic complexes. We refer
the reader to \cite{Bestvina_Feighn_free_factor,Handel_Mosher} for details. 
Theorem~\ref{theorem:linear_translation} applies
to all known actions. However, since in no case is the action known to be
acylindrical, Theorem~\ref{thm:hyperbolic_lower_bound} does not apply (as far as we know).
\end{remark}

We now prove Theorem \ref{theorem:convergence}.  Kaimanovich
\cite{Kaimanovich} showed that a random walk on a non-elementary
subgroup of isometries of a proper Gromov hyperbolic space $Y$
converges to the boundary $\partial Y$ almost surely. If $Y$ is
proper, then $Y \cup \partial Y$ is compact. However, the only place
where this is used is to show that there is a weak limit of the
convolution measures $\mu^{*n}$ which is a probability measure on $Y
\cup \partial Y$. Our initial task is therefore to show that for a
random walk on a non-elementary subgroup of isometries of a non-proper
Gromov hyperbolic space, the convolution measures $\mu^{*n}$ converge
to a $\mu$-invariant probability measure supported on the Gromov
boundary $\partial Y$.

\begin{lemma} \label{lemma:hitting measure}
Let $G$ be the isometry group of a (not necessarily proper) Gromov
hyperbolic simplicial complex $Y$, and let $\mu$ be a probability
distribution with finite support which generates a non-elementary
subgroup of $G$. Then there is a probability distribution $\nu$ on
the Gromov boundary $\partial Y$ which is a weak limit of the
convolution measures $\mu^{*n}$.
\end{lemma}

As the Gromov boundary is not compact, it will be convenient to
consider an alternative compactification, namely the horofunction
compactification, which we now describe.  Let $Y$ be a Gromov
hyperbolic simplicial space, which is not necessarily locally compact,
and let $C(Y, \mathbb{R})$ be the space of continuous functions on
$Y$, with the compact-open topology, which in this case is
equivalent to the topology of uniform convergence on compact sets.  As
$\mathbb{R}$ is Hausdorff, $C(Y, \mathbb{R})$ is also Hausdorff.  Let
$y_0 \in Y$ be a basepoint.  There is a map from $Y$ to $C(Y,
\mathbb{R})$, defined by sending $y$ to the corresponding
horofunction, $h_y(z) = d(z, y) - d(y_0, y)$. Let $Y_h$ be the closure
of $Y$ in $C(Y, \mathbb{R})$, which is sequentially compact, and hence
compact as $C(Y, \mathbb{R})$ is Hausdorff. The space $Y_h$ is known
as the horofunction compactification of $Y$, and $Y_h \setminus Y$ is
called the horofunction boundary.

We now define a ``local minimum'' map $\phi \colon Y_h \to Y
\cup \partial Y$.  Given a function $h \in Y_h$, consider $\inf(h) =
\inf_{y \in Y} h(y)$, which takes values in $[ -\infty, 0]$. If a
horofunction corresponds to a point in $Y$, i.e $h = h_y$ for some $y
\in Y$, then $\inf(h_y) = -d(y_0, y)$. We now consider the two cases
depending on whether $\inf(h) > - \infty$, or $\inf(h) = - \infty$.
If $\inf(h) > - \infty$, then there is a point $y$ such that $h(y) \le
\inf(h) + 1$, and we shall set $\phi(h) = y$. Note that if $h$ is
equal to $h_y$ for some $y \in Y$ then we may choose $\phi(h) = y$. If
$\inf(h) = - \infty$, then choose a sequence $y_n$ with $h(y_n) \to -
\infty$, and set $\phi(h)$ equal to the limit of $y_n$ in $\partial
Y$. Note that we can choose this sequence $y_n$ to be quasigeodesic,
and such that $h(y_n) = - n$.

\begin{proposition}
The local minimum map $\phi \colon Y_h \to Y \cup \partial Y$ is
coarsely well defined if $\inf(h) > - \infty$, and well defined if
$\inf(h) = - \infty$.
\end{proposition}

\begin{proof}
We start with some preliminary observations about horofunctions $h =
h_y$ corresponding to points $y \in Y$, i.e. $h_y(z) = d(z, y) -
d(y_0, y)$. In this case $h_y$ achieves its minimum value of $-d(y_0,
y)$ at the point $y$. Furthermore, for any geodesic $\gamma$, the
restriction of $h_y$ to $\gamma$ has a coarsely well defined minimum a
bounded distance away from the closest point projection $p$ of $y$ to
$\gamma$. The value of $h_y$ at $p$ is equal to $d(p, y) - d(y_0, y)$,
up to bounded error depending only on the constant of hyperbolicity
$\delta$, and for any other point $q \in \gamma$, the value of
$h_y(q)$ is equal to $d(p, q) + d(p, y) - d(y_0, y)$, up to bounded
error depending only on $\delta$, i.e.
\begin{equation} \label{eq:horofunction}
h_y(q) -K \le d(p, q) + d(p, y) - d(y_0, y) \le h_y(q) + K,
\end{equation}
for some constant $K$, depending only on $\delta$.

Let $y_1$ and $y_2$ be two points in $Y$ with $h(y_i) \le
\inf(h) + \epsilon$, and let $\gamma$ be a geodesic connecting
them. Let $y_n$ be a sequence of points in $Y$ such that the
corresponding horofunctions $h_{y_n}$ converge to $h$. As $\gamma$ is
compact, for any number $\epsilon > 0$ there is an $N$ such that for
all $n \ge N$, and for all points $y \in \gamma$, $\norm{h(y) -
  h_{x_n}(y)} \le \epsilon$. Let $p_n$ be the nearest point projection
of $y_n$ to $\gamma$, then by equation \eqref{eq:horofunction},
\[ h_{y_n}(p_n) \le \inf (h) + 2 \epsilon - d(y_1, y_2) / 2 + K, \]
where $K$ depends only on $\delta$, and so
$d(y_1, y_2)$ is bounded by a constant which only depends on $\delta$.

If $\inf(h) = - \infty$, there is a sequence of points $\{ y_n \}$
such that $h(y_n) \to - \infty$. We now show that this sequence $\{
y_n \}$ converges to a point in the Gromov boundary, and we shall map
$h$ to this point. Recall that a sequence $\{ y_n \}$ converges to the
boundary if for every number $B$ there is a constant $N$ such that the
Gromov product $(y_m \cdot y_n)_{y_0} \ge B$ for all $m \ge N$ and $n \ge
N$. So if the sequence $\{ y_n \}$ does not converge, then there is a
constant $B$ such that for all $N$ there are points in the sequence
$y_n$ and $y_m$ with $m \ge N$ and $n \ge N$ such that $(y_m \cdot
y_n)_{y_0} \le B$.

Let $\gamma$ be a geodesic from $y_n$ to $y_m$, and let $p$ be the
closest point on $\gamma$ to the basepoint $y_0$. The distance from $p$
to $y_0$ is equal to the Gromov product, up to an error which only
depends on $\delta$, so $d(y_0, p) \le B + K$, where $K$ only depends on
$\delta$. In particular this implies that the value of $h(p)$ is at
bounded below by $-B - K$. However, the value of $h$ at any point on
$\gamma$ is bounded above by the value of $h$ at the endpoints, up to
an error which only depends on $\delta$, and so as $h(y_n) \to -
\infty$, this implies that $h(p) \to -\infty$, which is a
contradiction. Therefore the sequence $y_n$ converges to a point in
the Gromov boundary.
\end{proof}

We now consider the $G$-equivariance of the local minimum map $\phi$.

\begin{proposition}
The local minimum map $\phi \colon Y_h \to Y \cup \partial Y$ is
coarsely $G$-equivariant if $\inf(h) > - \infty$, and $G$-equivariant
if $\inf(h) = - \infty$.
\end{proposition}

\begin{proof}
First suppose that $\inf(h) > - \infty$. From the definition of a
horofunction, $gh(z) = h(g^{-1} z) - h(g^{-1} y_0)$, so $\inf(g h) =
\inf(h) - h(g^{-1} y_0)$, so the action of $G$ preserves the set of
horofunctions with $\inf(h) > - \infty$. Furthermore, 
\begin{align*}
gh( g \phi(h) ) & = h( \phi(h) ) - h( g^{-1} y_0 ) \\
& \le \inf( gh ) + 1.
\end{align*}
Therefore both $g \phi(h)$ and $\phi(gh)$ have the property that their
values under $gh$ are within $1$ of $\inf(gh)$, and so they are a
bounded distance apart.  Therefore, $\phi$ is coarsely $G$-equivariant
on the set of elements with $\inf(h) > - \infty$.

Now suppose that $\inf(h) = - \infty$, and let $y_n$ be a sequence in
$Y$ such that $h_{y_n} \to h$. Recall that $\phi(h_{y_n}) = y_n$, and
$\phi(h)$ is the limit point of $y_n$ in the Gromov boundary. The
group $G$ acts by homeomorphisms on $Y \cup \partial Y$, so $g
\phi(h_{y_n}) = g y_n$, which converges to $g \phi(h)$. As $G$ also
acts by homeomorphisms on $Y_h$, the sequence $g h_{y_n}$ converges to
$g h$. Then $\phi(gh)$ is equal to the limit of $\phi( g h_{y_n} ) = g
y_n$, and this sequence converges to $g \phi(h)$, as
required. Therefore $\phi$ is $G$-equivariant on the set of $h$ such
that $\inf(h) = - \infty$.
\end{proof}

We now show that the map $\phi$ is measurable on the set of $Y_h$ with
$\inf(h) = - \infty$. In fact, it is continuous.

\begin{proposition}
The map $\phi \colon Y_h \to \partial Y$ is continuous on the set of
$h \in Y_h$ with $\inf(h) = - \infty$.
\end{proposition}

\begin{proof} 
Let $h_n$ be a sequence in $Y_h$, with $\inf(h_n) = \infty$, which
converges to a point $h \in Y_h$, also with $\inf(h) = - \infty$. It
suffices to show that $\phi(h_n)$ converges to $\phi(h)$.

For each $h_n$, there is a quasigeodesic sequence $(y_{n,m})$ with
$h_n(y_{n, m}) = -m$, and also for $h$ there is a quasi-geodesic
sequence $y_m$ with $h(y_m) = -m$. By the definition of $\phi$, the
sequence $\phi(h_n)$ is equal to the limit of $(y_{n, m})$ and
$\phi(h)$ is equal to the limit of $y_m$.

We now show that the Gromov product $( \phi(h_n)\cdot \phi(h) )_{y_0} \to
\infty$. It suffices to show that the quasigeodesics $(y_{n, m})_{m
  \in \mathbb{N}}$ and $(y_m)$ fellow travel on longer and longer
initial subsequences as $n \to \infty$. Suppose not, then no sequence
$(y_{n, m})$ fellow travels with $(y_m)$ past some given point, $y_p$
say. Consider the point $y_{q}$ for $q > p$. The function $h_n$ is
decreasing along the quasigeodesic $(y_{n, m})$, and has no maximum on
any geodesic. By thin triangles, any geodesic from $y_{n, m}$ to $y_q$
passes close to $y_q$, and as it is increasing from $y_{n, m}$ to
$y_q$, it must also be increasing from $y_p$ to $y_q$. In particular,
this implies that $h_n(y_q) \ge h_n(y_p)$ for all $n$. As $h_n \to h$,
the functions $h_n$ must in particular converge pointwise at $p$ and
$q$, so $h_n(y_p) \to h(y_p)$ and $h_n(y_q) \to h_n(y_q)$. This gives
a contradiction, as $h$ is decreasing along $y_m$, and $h(y_q) <
h(y_p)$. Therefore $\phi(x_n) \to \phi(x)$, so $\phi$ is continuous on
those $h \in Y_h$ with $\inf(h) = - \infty$, as required.
\end{proof}
 
Consider the convolution measures $\mu^{*n}$ on $Y_h$. As $Y_h$ is
compact, there is a weak limit $\nu$, which is a $\mu$-stationary
probability measure. The measure $\nu$ pushes forward to a
$\mu$-stationary probability measure on $Y \cup \partial Y$, which by
abuse of notation we shall also refer to as $\nu$. Actually this requires
some elaboration: the map from $Y_h$ to $Y \cup \partial Y$ is only coarsely
defined when the image is in $Y$, but is well-defined when the image is in
$\partial Y$. On the other hand, it will turn out as a consequence of the next lemma
that the part of $Y_h$ having image
in $Y$ has zero $\nu$-measure, and therefore the pushforward of $\nu$ to
$\partial Y$ is well-defined and $\mu$-stationary, as claimed.

\begin{lemma}[\cite{Maher_heegaard}, Lem.~3.5]\label{lemma:measure zero}
Consider a random walk generated by a symmetric finitely supported
probability distribution $\mu$ on the isometry group of a (not
necessarily proper) Gromov hyperbolic simplicial complex $Y$, such
that the group $G$ generated by the support of $\mu$ is
non-elementary. Let $\nu$ be a $\mu$-stationary probability measure on
$Y \cup \partial Y$, and let $\norm{\ \cdot\ }_G$ be any proper metric
on $G$.

Let $X$ be a set with the property that there is a sequence $\{ k_i
\}_{i \in \mathbb{N}}$ such that for any translate $g X$ of $X$ there
is a sequence $\{ w_i\}_{i \in \mathbb{N}}$ of elements of $G$, such
that the translates $gX, w_1gX, w_2gX, \ldots$ are all disjoint, and
$\norm{w_n}_G \le k_n$ for all $n$. Then $\nu(X) = 0$.
\end{lemma}

We shall choose $X = B(y_0, r)$, the ball of radius $r$ in $Y$.  As $G$
is non-elementary it contains a hyperbolic isometry $f$ with
translation length greater than $r + 2 \delta$. Then for any translate
$g B(y_0, r)$, the translates $f^n g B(y_0, r)$ are all disjoint. Choose
$\norm{\ \cdot \ }_G$ to be word length with respect to the generating
set consisting of the support of $\mu$. As $\norm{f^n}_G \le n
\norm{f}_G$, we may choose $k_n = n \norm{f}_G$, and then Lemma
\ref{lemma:measure zero} implies that $\nu(B(y_0, r)) = 0$. As this
holds for every $r$, this implies that $\nu(Y) = 0$, as required.
This completes the proof of Lemma \ref{lemma:hitting measure}, showing
that there is weak limit of the convolution measures supported on the
Gromov boundary.

We now prove Theorem \ref{theorem:convergence}, convergence to the
boundary, following the arguments of Kaimanovich
\cite{Kaimanovich}. The action of $G$ on $Y \cup \partial Y$ satisfies
the following two properties:

\begin{itemize}

\item[(CP)] If the sequence $g_n y_0$ converges to a point in $\partial
  Y$, then the sequence $g_n h y_0$ converges to the same point, for any
  $h \in G$.

\item[(CS)] The boundary $\partial Y$ consists of at least three
  points, and there is a $G$-equivariant Borel map $S$ assigning to
  pairs of distinct points $b_1, b_2$ in $\partial Y$ subsets (strips)
  $S(b_1, b_2) \subset Y$, such that for any three pairwise distinct
  points $b_i \in \partial Y$, $i=0,1,2$, there are neighborhoods
  $b_0 \in U_0 \subset Y$, and $b_i \in U_i \subset \partial Y$, $i =
  1, 2$ with the property that $S(b_1, b_2) \cap U_0 = \varnothing$
  for both $b_i \in U_i$, $i=1, 2$.

\end{itemize}

We have shown that $(\partial Y, \nu)$ is a $\mu$-boundary for
$G$. Furthermore, the action of $G$ on $Y \cup \partial Y$ satisfies
properties (CP) and (CS) from Kaimanovich \cite{Kaimanovich}, and so
all hypotheses of \cite{Kaimanovich} Thm~2.4 are satisfied, except
compactness. However, in the proof of Thm~2.4 compactness is only used
to guarantee the existence of a $\mu$-stationary probability measure,
and we have shown how to construct such a measure above. Therefore the
conclusion of Thm~2.4 holds in this setting, and this completes the
proof of Theorem \ref{theorem:convergence}.

The properties of exponential decay and linear progress now follow
from \cite{Maher_linear}, as these arguments only use
$\delta$-hyperbolicity and convergence to the boundary, which we have
now established. We may now complete the proof of Theorem
\ref{theorem:linear_translation}.

\begin{proof}(of Theorem \ref{theorem:linear_translation}) 
We briefly review some properties of quasigeodesics from
\cite{bridson_haefliger} Section III.H, see also Fujiwara
\cite{Fujiwara2} Section 1.2.  Let $\alpha$ be a path which is locally
quasigeodesic, i.e. there are constants $D, K$ and $c$ such that every
subpath of length at most $D$ is a $(K, c)$-quasigeodesic. Then there
is a constant $D_0$, which depends only on the constant of
hyperbolicity $\delta$, and the quasigeodesic constants $K$ and $c$,
such that for all $D \ge D_0$, the locally quasigeodesic path $\alpha$
is globally a $(K', c')$-quasigeodesic, where $K'$ and $c'$ depend
only on $\delta$.  A local quasigeodesic may be constructed by
concatenating geodesic segments with bounded overlaps. More precisely,
let $\alpha$ and $\beta$ be two geodesics in $Y$ such that the final
point of $\alpha$ is equal to the initial point of $\beta$, and let
$\gamma$ be a geodesic from the initial point of $\alpha$ to the final
point of $\beta$. We define the \emph{overlap} $\mathcal{O}(\alpha ,
\beta)$ of $\alpha$ and $\beta$ to be the largest distance from any
point on $\alpha \cup \beta$ to $\gamma$. There is a constant $Q$,
which only depends on $\delta$, such that if $\alpha$ is a path
consisting of a union of geodesics, each of length at least $Q$, such
that each successive pair overlap by at most $2 \delta$, then $\alpha$
is a quasigeodesic, with quasigeodesic constants depending only on
$\delta$.

Let $\gamma$ be a geodesic from $y_0$ to $g y_0$, and let
$\overline{\gamma}$ be $\gamma$ with the reverse orientation. Consider
the path formed from the union of the geodesic segments $g^k
\gamma$. Let $m$ be the midpoint of $\gamma$, and let $\alpha$ be a
geodesic from $m$ to $g m$, then the union of the geodesic segments
$g^k \alpha$ also forms a path in $Y$. Let $B$ be the overlap of
$\gamma$ and $g \gamma$. By thin triangles, there is a constant $K$,
which only depends on $\delta$, such that if the length of $\gamma$ is
at least $2B + Q + K$, then then length of $\alpha$ is at least $Q$,
and so the union of the $g^k \alpha$ forms a $(K', c')$-quasigeodesic,
where $K'$ and $c'$ only depend on $\delta$. Furthermore, the distance
between $m$ and $gm$ is equal to $\text{length}(\gamma) - 2B$, up to
an additive error which depends only on $\delta$, and this in turn is
equal to the translation length of $g$, again up to an additive error
which only depends on $\delta$.

If the size of of the overlap between $\gamma$ and $g \gamma$ is at
least $B$, then $\gamma$ has an initial segment of length $B$ which
fellow travels with a final segment of $g^{-1} \gamma$, which is the
initial segment of $\overline{\gamma}$ from $y_0$ to $g^{-1} y_0$. In
particular, this means that $g^{-1} y_0 \in S_{y_0}(g y_0, B + K )$,
for some constant $K$ which only depends on $\delta$.

We now estimate the probability that this occurs for a random walk of
length $n$.  Recall that by exponential decay, there are constants $K$
and $c_1 < 1$ such that for any $g \in G$
\[ \Pr ( w_n y_0 \in S_{y_0}(g y_0, r) ) \le K c_1^{d(y_0, g y_0) -
  r}. \]
As $\mu$ is symmetric, and choosing $r = 3Ln/4$, we obtain
\[ \Pr( w_n^{-1} y_0 \in S_{y_0}(w_n y_0,  \tfrac{3}{4} Ln) ) \le K
c_1^{d(y_0, w_n y_0) - \tfrac{3}{4} Ln}. \]
By linear progress, the probability that $d(y_0, w_n y_0) \le Ln$
decays exponentially in $n$, which implies
\[ \Pr( w_n^{-1} y_0 \in S_{y_0}(w_n y_0,  \tfrac{3}{4} Ln) ) \le K
c_1^{\tfrac{1}{4} Ln} + O(c_2^n), \]
where $c_2< 1$ is the exponential decay constant from linear progress.
Therefore the probability that the overlap of $\gamma$ and $w_n
\gamma$ is at most $Ln/4$ decays exponentially in $n$. As we have
shown that the translation length of $w_n$ is equal to $\length(\gamma)
- 2 \mathcal{O}(\gamma, \overline{\gamma})$, up to additive error
depending only on $\delta$, this implies that the probability that
translation length of $w_n$ is at least $Ln/2$ tends to one
exponentially fast, as required.
\end{proof}

\section{Universal lower bounds}\label{lower_bound_section}

In this section we shift our focus abruptly, and concentrate on obtaining uniform
{\em lower bounds} on $\scl$ for random walks in arbitrary finitely generated groups.
If $G$ is any group, then Bavard duality implies that
either $\scl$ vanishes identically on $[G,G]$, or else there is a homogeneous
quasimorphism $\phi$ on $G$. In the former case, there is nothing to say. In the latter
case, one obtains uniform lower bounds on $\scl$ from an estimate on the distribution
of values of $\phi$.

It turns out that there is a {\em central limit theorem} for $\phi$, proved by 
Bj\"orklund--Hartnick \cite{Bjorklund_Hartnick}. This theorem allows us to obtain
uniform lower bounds on $\scl$ of order $O(\sqrt{n})$. An important special case
of this theorem concerns quasimorphisms obtained from actions of groups on circles
(we obtained this special case independently of Bj\"orklund--Hartnick, though the
method of proof is similar). In the next section we prove the central limit theorem
for circle actions, and derive some geometric applications, of independent
interest.

\subsection{Groups acting on circles}\label{circle_clt_subsection}

Let $\homeo^+(S^1)$ denote the group of orientation-preserving homeomorphisms of the
circle. This group has a universal central extension consisting of the 
group of homeomorphisms of $\R$ that commute with integer translation; we 
denote this group $\homeo^+(\R)^\Z$.

Poincar\'e defined a function $\rot:\homeo^+(\R)^\Z \to \R$ called {\em rotation
number} by the formula
$$\rot(h) = \lim_{n \to \infty} \frac {h(n)} n$$
This descends to a function $\rot:\homeo^+(S^1) \to \R/\Z$.
The function $\rot$ is a homogeneous quasimorphism on $\homeo^+(\R)^\Z$ with
defect $D(\rot)=1$.

Now let $G$ be any group. Suppose $G$ acts on the circle
by orientation-preserving homeomorphisms; i.e.\/ suppose we have $G \to \homeo^+(S^1)$.
The preimage of $G$ in $\homeo^+(\R)^\Z$ is a (possibly split) central extension
$\overline{G}$ of $G$ and we obtain a homogeneous quasimorphism $\rot$ on $\overline{G}$.
The defect of $\rot$ on $\overline{G}$ is usually equal to $1$, but might be smaller,
for instance if the centralizer of $\overline{G}$ in $\homeo^+(\R)^\Z$ is bigger
than the center $\Z$.

Let $\overline{G} \to \homeo^+(\R)^\Z$ be some representation as above, and let $\overline{S}$
be a finite generating set for $\overline{G}$ (note that $\overline{G}$ is 
finitely generated if and only if $G$ is). Let $\mu$ be the uniform probability
measure on $\overline{S}$, and let $\mu_n$ be the $n$-fold convolution of $\mu$. That is, $\mu_n$
is the probability measure associated to a random walk of length $n$ on $G$ in the
generators $\overline{S}$. Note that we do {\em not} require $S$ to be symmetric or to generate
$G$ as a group.

\begin{theorem}[Central limit theorem]\label{circle_clt}
Let $\overline{G}$ be a subgroup of $\homeo^+(\R)^\Z$, and let $\overline{S}$ be a finite 
generating set. Let $g_0,g_1,\cdots$ be a Markov process on $G$,
where $g_0=\id$, and where each $g_n$ is obtained from $g_{n-1}$ by right
multiplication by a random element $s \in \overline{S}$ (in the
uniform measure). Then there
is a central limit theorem for $\rot$; i.e\/ there are constants $E$ and $\sigma$
so that $n^{-1/2}(\rot(g_n) - En)$ converges in probability to the 
Gaussian measure $N(0,\sigma)$.
\end{theorem}
\begin{proof}
Let $G$ be the image of $\overline{G}$ in $\homeo^+(S^1)$, and let $S$ be the
image of $\overline{S}$. Further, let $F$ be the semigroup generated by $S$ (which
might be smaller than $G$).
Let $M(S^1)$ denote the space of probability measures on $S^1$,
and let $\Delta_S:M(S^1) \to M(S^1)$ be defined by
$$\Delta_S(\nu) = \frac 1 {|S|} \sum_{s \in S} s_*\nu$$
A fixed point of $\Delta_S$ is called a {\em stationary} (or {\em harmonic}) 
measure on $S^1$. We let $\nu$ be an ergodic stationary measure. By construction,
the action of $F$ on $S^1$ is absolutely continuous with respect to $\nu$.

If $\nu$ contains atoms, then there is an atom of biggest measure, supported
at some point $p$. Since $\nu(p)$ is equal to the average of $\nu(sp)$, it follows
that the measure of every $\nu(sp)$ is equal to the measure of $\nu(p)$. It follows
that every point in $Fp$ has the same (atomic) measure, and therefore this set is
finite and $F$-invariant. Since the action is by homeomorphisms, this 
set is invariant under $G$, and therefore $G$ preserves a probability measure on
$S^1$. In this case it is well-known that $\rot$ is a {\em homomorphism} from
$\overline{G}$ to $\R$, and the ordinary central limit theorem applies. So we assume
in the sequel that $\nu$ contains no atoms.

Since $\nu$ is stationary, for any measurable $I\subset S^1$) (in particular, for
every interval $I$), there is an equality
$$\frac 1 {|S|} \sum_{s\in S} \nu(S(I)) = \nu(I)$$
Let $\overline{\nu}$ be the Radon measure
on $\R$ obtained by identifying $\R$ locally with $S^1$. Observe that $\overline{\nu}$
is invariant under integer translation, and furthermore it satisfies
$\overline{\nu}([t,t+1])=1$ for all $t\in \R$. Furthermore, $\overline{\nu}$ is
evidently stationary for $\overline{S}$; i.e.\/ $\frac 1 {|S|} \sum_{s\in \overline{S}}
\overline{\nu}(s(I)) = \overline{\nu}(I)$ for all measurable $I\subset \R$.

Define $F:\R \to \R$ as follows. For each $t$ choose $T\ll 0$ and $T\ll t$, and define
$$F(t) = \overline{\nu}[T,t] - \overline{\nu}[T,0]$$
(evidently, $F$ does not depend on the choice of sufficiently negative $T$).
For any $t\in \R$, define
$$f(t) = \left( \frac 1 {|S|} \sum_s F(s(t)) \right) - F(t)$$
If $u$ is arbitrary, and $I$ is the interval with extremal points $t$ and $u$, then
$$f(u) = f(t) + \left( \frac 1 {|S|} \sum_s \overline{\nu}(s(I)) \right) - \overline{\nu}(I) = f(t)$$
Hence $f(\cdot)$ is {\em constant}, and equal to some fixed $E$, which we call
the {\em drift} of $S$.

It follows from this that the function $F(g_n(0))-nE$ is a (bounded) 
{\em martingale}; that is, the expected value of $F(g_n(0))-nE$ given 
$g_{n-1}$ is $F(g_{n-1}(0))-(n-1)E$.

Now, if $X_i$ is any martingale with bounded increments, if
$\sigma_i^2$ is the expectation of $(X_{i+1}-X_i)^2$ given $X_1,X_2,\cdots,X_i$,
and if $\tau_n$ is the minimum $n$ so that $\sum_{i=1}^n \sigma_i^2 \ge n$, then
the Martingale central limit theorem (see e.g.\/ \cite{Hall_Heyde}, especially
Thm.~3.2. on page 58) says that 
$X_{\tau_n}/\sqrt{n}$ converges in probability to a normal distribution $N(0,1)$.
In our particular case, it turns out that $\tau_n/n$ converges in probability
to $\sigma^2$ for some constant $\sigma$. For, the expectation of $(F(g_n)-F(g_{n-1})-E)^2$
given $g_{n-1}$ depends only on $g_{n-1}(0)$ mod $\Z$, and since by hypothesis
the measure $\nu$ is ergodic for the action of $\Delta_S$, the random
ergodic theorem (see \cite{Furman}, Thm.~3.1) implies such convergence in
probability (even in $L^1$).

Now, $|\rot(g)-g(0)|\le 1$ for any $g\in \overline{G}$, and moreover $|F(t)-t|\le 1$.
Therefore a central limit theorem for the function $F(g_n(0))-nE$ implies one for $\rot(g_n)-nE$
and the theorem is proved.
\end{proof}

\begin{remark}\label{symmetric_drift_vanishes}
If $S=S^{-1}$ (symmetric random walk in a group), then the random process is
invariant under taking inverses, and therefore $E=0$. Similarly, if
$S$ is conjugated to itself (in $\homeo^+(\R)^\Z$) by some reflection 
$t \to 2C-t$, then $E=-E=0$ by symmetry.
\end{remark}

\begin{corollary}\label{scl_sqrt_corollary}
Let $G$ be a finitely generated subgroup of $\homeo^+(S^1)$, and let
$\overline{G}$ be the preimage in $\homeo^+(\R)^\Z$. Suppose further that
$\scl$ vanishes identically on $G$. Then if $\scl_n$ denotes the value of
$\scl$ on a random walk in $\overline{G}$ (in some finite symmetric generating
set), there is some $\sigma$ for which the following is true:
$$\lim_{n \to \infty} \Pr(a< \scl_n/\sigma \sqrt{n} < b) = \frac 2 {2\pi} \int_a^b \1_{[0,\infty)}
e^{-x^2/2} dx$$
\end{corollary}
\begin{proof}
For any group $G$ there is a short exact sequence
$$0 \to H^1(G) \to Q(G) \to H^2_b(G) \to H^2(G)$$
where $H^2_b$ denotes bounded cohomology. Since $\scl$ vanishes identically
on $G$ by hypothesis, $Q(G)/H^1(G)=0$ by Bavard duality; i.e.\/ Theorem~\ref{Bavard_duality_theorem}.
However, the coboundary $\delta\rot$ exists as an element of $H^2_b(G)$ whose
image in $H^2(G)$ is nontrivial, and equal to the familiar {\em Euler class}.

Since $\overline{G}$ is the central extension associated to the Euler class, it
follows that $Q(\overline{G})/H^1(\overline{G})$ is one dimensional, and spanned
by $\rot$. So $\scl(g)=|\rot(g)|/2D(\rot)$ on the commutator subgroup of
$\overline{G}$, and the conclusion follows.
\end{proof}

\begin{example}
Corollary~\ref{scl_sqrt_corollary} applies to many naturally occurring families of groups, including
Hilbert modular groups $\SL(2,\O(n))$ where $\O(n)$ is the ring of integers in $\Q(\sqrt{n})$ for
$n$ square-free, $\SL(2,\Z[1/2])$, Thompson's circle group $F$ and certain generalized Stein-Thompson groups, 
and many others. The fact that $\scl$ vanishes identically on these groups follows from the stronger property that
they are {\em boundedly generated by commutators}. 

For Hilbert modular groups, this is a consequence of a deep theorem of
Carter--Keller--Paige, namely \cite{Carter_Keller_Paige} Thm.~6.1 which says that if $A$ is
the ring of integers in a number field $K$ containing infinitely many units, and $T$ is an element
of $\SL(2,A)$ which is not a scalar matrix, then $\SL(2,A)$ has a finite index normal subgroup which
is boundedly generated by conjugates of $T$.

The case $\SL(2,\Z[1/2])$ is due to Liehl \cite{Liehl} who proves that the group is boundedly
generated by elementary matrices (which are themselves products of commutators of bounded length).

The case of Thompson's group is due to Ghys--Sergiescu \cite{Ghys_Sergiescu}, and some generalizations
are due to Zhuang \cite{Zhuang}. For an introduction to Thompson's groups and their properties, see
\cite{Cannon_Floyd_Parry}.
\end{example}

\subsection{Random turtles in the hyperbolic plane}

In this subsection we give a geometric application of
Theorem~\ref{circle_clt} of independent interest. Consider the following random process.
A turtle starts at the origin in the hyperbolic plane, and moves by alternately moving forward
some fixed distance $\ell$, and by turning either left or right through some fixed angle 
$\alpha$. Let $p_0,p_1,\cdots$
denote the locations of the turtle after each successive move forward. Note $p_0$ is the origin,
and $d(p_i,p_{i+1})=\ell$ for each $i$. We think of the $p_i$ as the vertices of a random
polygonal path.

For each $n$, let $P_n$ be the polygon with $n+1$ (cyclic) vertices $p_0,p_1,\cdots,p_n$.
There are (at least) two natural geometric quantities to associate to $P_n$. If $\gamma:S^1 \to \R^2$
is a $C^1$ immersion, the {\em winding number} is the degree of the Gauss map 
$\theta \to \gamma'(\theta)/|\gamma'(\theta)|$. Moreover, the {\em algebraic area} enclosed
by $\gamma$ is the integral $\int_{\R^2} \wind(\gamma,p) d\area(p)$, where the {\em local 
winding number} $\wind(\gamma,p)$ is the degree of the map $\theta \to (\gamma(\theta)-p)/|(\gamma(\theta)-p)|$.

Although $\partial P_n$
is only piecewise linear, it can be smoothed canonically by rotating the tangent vector left or right
through an angle $\alpha$ at each vertex $p_i$ with $0<i<n$
(according to the behavior of the turtle), and then
in an arbitrary way at $p_0$ and $p_n$. Thus we can assign to the sequence $p_n$ two geometric
quantities $W_n$, the winding number of $\partial P_n$, and $A_n$, the algebraic area enclosed by
$\partial P_n$.

Let $\alpha_i$ be the signed turning angle of $P_n$ at the vertex $p_i$. Then the 
Gauss--Bonnet theorem for immersed polygons says there is an equality
$$2\pi W_n -  \sum_i \alpha_i = - A_n$$

\begin{theorem}[Area and Winding Theorem]\label{random_area_theorem}
Fix some angle $\alpha$ and length $\ell$. Let $P_n$ be a random polygon in the hyperbolic plane
with (cyclic) vertices $p_0,p_1,\cdots,p_n$, where $d(p_i,p_{i+1})=\ell$ for each $0\le i\le n-1$
and an angle of $\pm \alpha$ at each $p_i$ with $0<i<n$, with signs independently and
uniformly chosen from $\pm 1$. Let $A_n$ be (as above) the algebraic area enclosed by $P_n$,
and $W_n$ the winding number of $\partial P_n$. Then
$A_n$ and $W_n$ both satisfy a central limit theorem with mean $0$.
\end{theorem}
\begin{proof}
Let $L$ and $R$ be the hyperbolic isometries which translate the origin a distance $\ell$ and then
rotate either left or right through angle $\alpha$. As matrices in $\PSL(2,\R)$, we can take
$$R=\begin{pmatrix} e^{\ell/2}\cos(\alpha/2) & e^{\ell/2}\sin(\alpha/2) \\ -e^{-\ell/2}\sin(\alpha/2) & e^{-\ell/2}\cos(\alpha/2) \end{pmatrix}
\quad
L= \begin{pmatrix}e^{\ell/2} \cos(\alpha/2) & -e^{\ell/2}\sin(\alpha/2) \\ e^{-\ell/2}\sin(\alpha/2) & e^{-\ell/2}\cos(\alpha/2) \end{pmatrix}$$

The successive $p_i$ are obtained from $p_0$
by right multiplication by a random sequence of $R$'s and $L$'s. We think of $p_i = w_ip_0$ where $w_i$
is an element of the free semigroup generated by $R$ and $L$.
By exponentiation, the
unit tangent circle at each point in $\H^2$ is canonically identified with the ideal circle $S^1_\infty$
(this defines a canonical flat projective connection on the unit tangent bundle). Therefore,
the winding number can be computed from the (lifted) action on the ideal circle. We think of
$R$ and $L$ as elements of $\homeo^+(\R)^\Z$, normalized to move some basepoint distance $<1$ in
the positive and negative directions respectively. Then $|W_i - \rot(w_i)|<1$, so 
Theorem~\ref{circle_clt} shows that $W_n$ satisfies a central limit theorem.
Note that $E=0$, by the left-right symmetry (see Remark~\ref{symmetric_drift_vanishes}).

With notation as in the proof of Theorem~\ref{circle_clt}, the central limit theorem for $W_n$
follows from the martingale property of the function $F(w_n(0))$. However, the function
$F(w_n(0)) - \sum_{i=1}^n \alpha_i/2\pi$ is also a martingale, so the same proof gives a central
limit theorem for $A_n$, as claimed.
\end{proof}

\begin{remark}\label{left_right_remark}
If $\alpha$ is small enough
compared to $\ell$, the polygonal path $p_0,p_1,p_2,\cdots$ is uniformly quasigeodesic. 
In this case, the areas of successive triangles $p_0,p_n,p_{n+1}$
define a H\"older continuous function on the one-sided shift space on the alphabet
$\lbrace L,R \rbrace$. The H\"older continuity amounts to the observation that for four points $a,b,c,d$
in $\H^2$ with $d(a,b)=d(c,d)=\ell$, the difference of $\area(a,c,d)$ and $\area(b,c,d)$ is bounded
by a constant that decays exponentially fast in the distance from $b$ to $c$. The central limit theorem
for H\"older functions on shift spaces gives a different proof in this case. This kind of argument is 
implicit in \cite{Picaud} and a related argument is pursued in \cite{Calegari_Fujiwara, Horsham_Sharp}.
\end{remark}

\begin{remark}
There is nothing very special (apart from its charm) about the particular random model we chose
for the polygons $P_n$. We could just as easily fix some finite subset $S \subset \PSL(2,\R)$ and
define a random sequence $p_i$ by $p_i = sp_{i-1}$ for some random $s\in S$. The winding number
and algebraic area of $P_n$ satisfy a central limit theorem in this case too, and with essentially
the same proof. Note in this generality, the drift might be nonzero.
\end{remark}

\begin{remark}
It is interesting to study how the statistical quantities associated to $A_n$ and $W_n$ vary as
a function of the parameters. Let's fix $\alpha$ and let $\ell$ vary, and consider the random
winding number $W_n$. By the left-right symmetry, the drift $E(\ell)$ is zero for all $\ell$. 
But the standard deviation $\sigma(\ell)$ undergoes a phase transition: it is zero for 
$\ell \in [2\cosh^{-1}(1/\sin(\alpha/2)),\infty)$, and increases monotonically to $\alpha$ as 
$\ell$ decreases to $0$. In this case $\sigma(\ell)$ is real analytic (as a function of $\ell$) 
on $[0,2\cosh^{-1}(1/\sin(\alpha/2)))$.
\end{remark}

\subsection{Central limit theorem for arbitrary quasimorphisms}

In fact, very shortly after proving Theorem~\ref{circle_clt}, we learned that a completely general 
statement has independently been obtained by Bj\"orklund--Hartnick \cite{Bjorklund_Hartnick}. 
They prove the following theorem (in fact, their results hold in considerably greater generality):

\begin{theorem}[Bj\"orklund--Hartnick \cite{Bjorklund_Hartnick}, Thm.~1.1]\label{BH_clt_theorem}
Let $G$ be a finitely generated group, and $S$ a finite generating set. Let
$f:G \to \R$ be a quasimorphism, and $X_n:=s_n \cdots s_1$ an i.i.d. left-random walk on $G$ in
the generating set $S$. Then $f$ satisfies a central limit theorem with respect to $X_n$. Moreover,
if the homogenization of $f$ is nonzero, then the central limit is non-degenerate. 
\end{theorem}

By Bavard Duality we immediately conclude the following:

\begin{corollary}\label{sqrt_n_gap}
Let $G$ be any finitely generated group, and suppose $H_1(G)$ is finite.
Let $S$ be a finite symmetric generating set. 
Suppose $Q(G)$ is finite dimensional but nonzero.  Let $g_n$ be obtained by random walk of length $n$
with respect to the uniform measure on $S$. Then 
for any $\epsilon$ there are positive constants $a,b$ depending on $\epsilon$
such that $\Pr(a < \scl(g_n)/\sqrt{n} < b) \ge 1-\epsilon$ for $n\gg 0$.
\end{corollary}

Theorem~\ref{BH_clt_theorem} generalizes Theorem~\ref{circle_clt} considerably, but the proof
turns out not to be too much harder. Under very general conditions, Bj\"orklund--Hartnick show 
that a quasimorphism $f$ on a group $G$ with a probability measure $\mu$ (satisfying some conditions)
has a {\em (bi-)harmonic representative}; i.e.\ there is a function $f'$ that differs from
$f$ by a bounded amount, and with the property that $f'$ is invariant under convolution with $\mu$
(the existence of such a harmonic representative was also proved in quite a different way by 
Burger--Monod \cite{Burger_Monod}).
It is this harmonic feature of $f'$ that lets one prove a CLT using the martingale CLT, as in
the proof of Theorem~\ref{circle_clt}.

\begin{remark}
We would like to point out that in the first version of their paper, Bj\"orklund--Hartnick
required that the probability measure on $S$ be {\em symmetric}, whereas we never required this
as a hypothesis.
\end{remark}

\begin{remark}
Burger--Monod and Bj\"orklund--Hartnick both work with individual quasimorphisms.
A finite collection of quasimorphisms $\phi_1,\phi_2,\cdots,\phi_m$ can be put together
into a single function $\Phi:G \to \R^m$ whose coordinates are the $\phi_i$. It makes
sense to ask for some harmonic representative $\Phi'$ with $|\Phi' - \Phi|<\infty$ and
a CLT for $\Phi'$ (and therefore also for $\Phi$). In fact, the arguments 
in \cite{Burger_Monod, Bjorklund_Hartnick} easily generalize to this situation. The
existence of a harmonic representative is essentially elementary. The quasimorphism
property implies by definition that there is a constant $C$ so that for any $g,h$
we have $|1/2(\Phi(gh) + \Phi(gh^{-1}))-\Phi(g)| \le C$. It follows that for $\mu$
symmetric, the $n$-fold convolutions $\Phi_n(g):=\int_G \Phi(gh)d\mu^{*n}(h)$ lie in a precompact family, and
therefore we obtain in the usual way a fixed point for convolution with $\mu$; i.e.\/ a harmonic
representative. Following Bj\"orklund--Hartnick we deduce a CLT 
for $\Phi$.
\end{remark}

The condition in Corollary~\ref{sqrt_n_gap}
that $H_1(G)$ is finite is an annoying technical restriction. The problem is that
one cannot perform the ``homological correction trick'' that we used earlier for two reasons.
Firstly, there is no compression of the support of a Gaussian measure, and therefore conditioning
on a rare event can change the order of magnitude of the distribution. Secondly, the order of
magnitude of the error term (i.e.\/ the size of the image in homology) is the same as the
term we are trying to control, so it is impossible to naively
obtain lower bounds, since the two terms might cancel. Nevertheless, the following conjectures
seem reasonable. The first just asserts that the hypothesis that $H_1(G)$ is finite in
Corollary~\ref{sqrt_n_gap} can be removed:

\begin{conjecture}[Local limit conjecture]
Let $G$ be any finitely generated group.
Let $S$ be a finite symmetric generating set. 
Suppose $Q(G)$ is finite dimensional but nonzero. Let $g_n$ be obtained by random walk of length $n$
with respect to the uniform measure on $S$, conditioned to lie in $[G,G]$. Then 
for any $\epsilon$ there are positive constants $a,b$ depending on $\epsilon$
such that $\Pr(a < \scl(g_n)/\sqrt{n} < b) \ge 1-\epsilon$ for $n \gg 0$.
\end{conjecture}

The second conjecture is more ambitious, and asserts that the $O(\sqrt{n})$ growth rate
should characterize (finitely presented)groups with $Q(G)$ finite dimensional:

\begin{conjecture}[Finite dimensionality]
Let $G$ be a finitely presented group, and let $g_n$ be obtained by random walk of length $n$
conditioned to lie in $[G,G]$. Suppose that for any $\epsilon$ there is a positive constant $b$
so that
$$\Pr(\scl(g_n)/\sqrt{n} < b) \ge 1-\epsilon$$
for $n\gg 0$.
Then $Q(G)$ is finite dimensional.
\end{conjecture}

Finally, we ask for finitely presented groups for which $\scl$ has {\em intermediate growth}:

\begin{question}\label{intermediate_growth_question}
Is there a finitely presented group $G$ so that if $g_n$ is obtained by random walk of length $n$
conditioned to lie in $[G,G]$, then for any $C>0$ and $\epsilon>0$, there is an estimate
$$\Pr(\scl(g_n) < C\sqrt{n} \text{ or } \scl(g_n) > C^{-1} n/\log{n}) < \epsilon$$ for
$n\gg 0$?
\end{question}

We believe that it should be possible to produce finitely {\em generated} groups with the 
property sought by Question~\ref{intermediate_growth_question}, by a careful small 
cancellation argument. It would also be very interesting to find examples of (uncountable)
{\em transformation groups} $G$ containing finitely generated subgroups $\Gamma$ so that the
growth rate of $\scl_G$ on random walk in $\Gamma$ is intermediate between $\sqrt{n}$ and
$n/\log{n}$.

\begin{question}
Let $\Sigma$ be a closed surface with $\chi(\Sigma)<0$, equipped with a smooth area form.
Let $G$ be the group of Hamiltonian diffeomorphisms of $\Sigma$. Is there a well-defined growth
rate of $\scl_G$ on random walk in $\Gamma$ for $\Gamma$ a ``generic'' finitely generated
subgroup of $G$? Does it grow like $n/\log{n}$?
\end{question}

\section{Acknowledgments}
Danny Calegari was supported by NSF grants DMS 0707130 and DMS
1005246. Joseph Maher was supported by NSF grant DMS 0706764 and
PSC-CUNY Award 60019-43-20, and would like to thank the Hausdorff
Research Institute for Mathematics for its hospitality during work on
this paper.  We would like to thank Jason Behrstock, Michael
Bj\"orklund, Sean Cleary, David Fisher, Alex Furman, Bradley Groff, Anders Karlsson,
Nikolai Makarov, Curt McMullen, Andr\'es Navas, Jay Rosen, Richard Sharp, Brian Simanek,
Alessandro Sisto, Alden Walker and the anonymous referee 
for some useful conversations about this material.

\end{document}